\documentclass[12pt,reqno]{amsart}
\usepackage{graphicx}
\usepackage{epstopdf}
\usepackage{url}
\usepackage{enumitem}
\usepackage{latexsym}
\usepackage{amscd, amsfonts,amsthm, amsmath,amssymb}
\newtheorem{theorem}{Theorem}[section]

\newtheorem{proposition}[theorem]{Proposition}
\newtheorem{prop}[theorem]{Proposition}
\newtheorem{lemma}[theorem]{Lemma}
\newtheorem{lem}[theorem]{Lemma}
\newtheorem{definition}[theorem]{Definition}
\newtheorem{defi}[theorem]{Definition}

\newtheorem{coro}[theorem]{Corollary}

\newtheorem{example}[theorem]{Example}
\newtheorem{remark}{Remark}

\newtheorem{fact}[theorem]{Fact}
\usepackage{hyperref}
\hypersetup{
    colorlinks=true,
    linkcolor=blue,
    filecolor=magenta,      
    urlcolor=cyan,
    pdfpagemode=FullScreen,
    }
\usepackage{cite}

\def\as{\mathrm{as}}

\def\Lpk{\mathrm{Lpk}}

\def\LPV{\mathrm{Lpv}}

\def\Val{\mathrm{Val}}
\def\Lpv{\mathrm{Lpv}}
\def\S{{\mathfrak S}}
\def\B{{\mathfrak B}}
\def\D{{\mathfrak D}}

\def\run{\textrm{run}\,}
\def\des{\textrm{des}}

\def\Neg{\mathrm{Negs}}
\def\leaf{\mathrm{leaf}}

\def\red{\textcolor{red}}

\def\Orb{\mathrm{Orb}}
\def\hrn{\mathcal{HR}_n}
\def\hr{\mathcal{HR}}
\def\bhr{\mathcal{BHR}}
\def\dhr{\mathcal{DHR}}
\def\A{\mathcal{A}}
\date{\today}
\usepackage{xcolor}%

\usepackage{tikz}
\usetikzlibrary{arrows.meta,positioning,calc,fit}
\usepackage{forest}
\usepackage{tikz}
\usetikzlibrary{trees,arrows.meta,positioning}
\usepackage{tikz}
\usetikzlibrary{positioning}
\usepackage{xcolor}

\tikzset{
  leaf/.style   ={circle,draw,minimum size=6mm,inner sep=1pt},
  minnode/.style={leaf, fill=blue!12, draw=blue!60!black},
  maxnode/.style={leaf, fill=red!12,  draw=red!60!black},
  note/.style   ={font=\scriptsize},
}

\usepackage{tikz}
\usepackage{forest}
\forestset{
  mytree/.style={
    for tree={
      draw, circle, inner sep=1.5pt, s sep=7mm, l sep=8mm,
      anchor=north
    }
  },
  highlight/.style={fill=gray!15, draw=black}
}

\begin{document}

\title[Counting alternating runs via Hetyei-Reiner trees]{Counting permutations by alternating runs via Hetyei-Reiner trees}
\author{Qiongqiong Pan, Yunze Wang, and Jiang Zeng}
\address{College of Mathematics and Physics, Wenzhou University\\
Wenzhou 325035, PR China}
\email{qpan@wzu.edu.cn}
\address{College of Mathematics and Physics, Wenzhou University\\
Wenzhou 325035, PR China}
\email{yzwang@stu.wzu.edu.cn}
\address{Universite Claude Bernard Lyon 1, CNRS, Centrale Lyon, INSA Lyon, 
Université Jean Monnet, ICJ UMR5208, 69622 Villeurbanne, France}
\email{zeng@math.univ-lyon1.fr}

\begin{abstract}
The generating polynomial of all $n$--permutations with respect to the number of alternating runs possesses a root at $-1$ of multiplicity $\lfloor (n-2)/2 \rfloor$ for $n \ge 2$. 
This fact can be deduced by combining the David--Barton formula for Eulerian polynomials with the Foata--Schützenberger $\gamma$--decomposition of these polynomials. 
Recently, Bóna provided a group--action proof of this result. 
In the present paper, we propose an alternative approach based on the Hetyei--Reiner action on binary trees, which yields a new combinatorial interpretation of Bóna’s quotient polynomial. 
Furthermore, we extend our study to analogous results for permutations of types~B and~D. 
As a consequence of our bijective framework, we also obtain combinatorial proofs of David--Barton type identities for permutations of types~A and~B.
\end{abstract}

\maketitle
\tableofcontents

\section{Introduction}
Let $\pi=\pi_1\pi_2\cdots\pi_n$ be a permutation in the symmetric group~$\mathfrak{S}_n$.
We say that $\pi$ \emph{changes direction} at an index $i\in\{2,\ldots,n-1\}$ if
\[
\pi_{i-1}<\pi_i>\pi_{i+1}
\qquad\text{or}\qquad
\pi_{i-1}>\pi_i<\pi_{i+1}.
\]
In other words, $\pi$ changes direction at~$i$ whenever a local rise is followed by a fall,
or vice versa.
The number of \emph{alternating runs} of~$\pi$, denoted $\run(\pi)$,
is defined as one plus the number of indices at which $\pi$ changes direction.
Equivalently, $\pi$ has $k$ alternating runs if there are $k-1$ indices where $\pi$ changes direction.

\begin{example}
The permutation
\[
\pi=7\,3\,2\,5\,6\,9\,1\,4\,8
\]
changes direction at positions $i=3,6,7$,
and thus has $4$ alternating runs:
\[
732,\quad 2569,\quad 91,\quad 148.
\]
Hence $\run(\pi)=4$.
\end{example}

For each $n\ge1$, the \emph{run polynomial} $R_n(x)$ is defined by
\begin{equation}\label{def: run polynomials}
    R_n(x)\;=\;\sum_{\pi\in \mathfrak{S}_n} x^{\run(\pi)}.
\end{equation}
The coefficients of $R_n(x)$ form the sequence
\href{https://oeis.org/A059427}{A059427} in the \emph{On-Line Encyclopedia of Integer Sequences}~\cite{oeis22}.
It is well known~\cite{BE00,Ma13,St08, Wi99} that
the polynomial $R_n(x)$ has a factor $(1+x)^m$, where
\begin{equation} \label{multiplicity-m}
m=\bigl\lfloor\tfrac{n}{2}\bigr\rfloor-1, \qquad n\ge2.
\end{equation}
A quick proof of this fact can be obtained by combining
a formula of David--Barton with the
$\gamma$--expansion formula for the Eulerian polynomials
due to Foata and Schützenberger (see Section~\ref{link to eulerian polynomials}).

\smallskip
In a recent paper, Bóna~\cite{Bo21} gave a non-inductive proof
of this divisibility by constructing an abelian group
isomorphic to~$\mathbb{Z}_2^{\,m}$ that acts on~$\mathfrak{S}_n$
so that the orbits are all of size~$2^m$.
Bóna also characterized the quotient
\[
\frac{R_n(x)}{(1+x)^m}
\]
as the enumerative polynomial of a certain family of
ad hoc permutations with respect to the number of alternating runs,
and suggested finding other interpretations of this quotient.

\smallskip
In the present paper, we take an alternative approach to Bóna’s problem
by using a group action on the $\min-\max$ trees
introduced by Hetyei and Reiner~\cite{HR98},
and further developed by Foata and Han~\cite{FH01};
see also~\cite[§1.6.3]{St97}.
The Hetyei--Reiner group action has recently found
successful applications in proving
the $\gamma$--positivity of bivariate Eulerian polynomials,
see~\cite{Su21,PZ23}.

\subsection*{Organization of the paper.}
The remainder of this paper is organized as follows.
In Section~\ref{sec:HR--action} we recall the Hetyei--Reiner action
(HR--action) on min--max trees.
In Section~\ref{sec:typeA} we apply the HR--action to study
the longest alternating subsequences and the alternating runs in permutations,
and provide a combinatorial explanation of their connection
with the $\gamma$--positive expansion of the Eulerian polynomials.
In Section~\ref{sec:typeB} we introduce a modified HR--action
to handle alternating runs of type~B permutations.
Finally, in Section~\ref{sec:typeD} we study the alternating runs
in type~D.

\section{Hetyei-Reiner action}\label{sec:HR--action}

\begin{definition}[Foata–Han~\cite{FH01}]
A $\min-\max$ tree is a rooted binary tree that satisfies the following conditions:
\begin{enumerate}[label=(\roman*)]
\item it is \emph{rooted};
\item it is \emph{binary}, i.e., each node has at most two children;
\item it is \emph{topological}, i.e., the left child is distinguished from the right child;
\item it is \emph{labeled}, that is, its nodes are in bijection with a finite totally ordered set of labels;
\item it satisfies the \emph{min--max property}:
for every node $x$, the label of $x$ is either the minimum or the maximum of all labels appearing in the subtree rooted at $x$.

A min--max tree is called a \emph{Hetyei–Reiner tree} (or \emph{HR--tree}) if, in addition, it satisfies the following \emph{HR-property}:
\item
whenever $s$ is an inner node whose label is the minimum (respectively, maximum) among the labels of its subtree $T(s)$ (rooted at $s$), the right subtree $T^r(s)$ (rooted at the right child of $s$) is nonempty and contains the node whose label is the maximum (respectively, minimum) in $T(s)$.
\end{enumerate}
\end{definition}

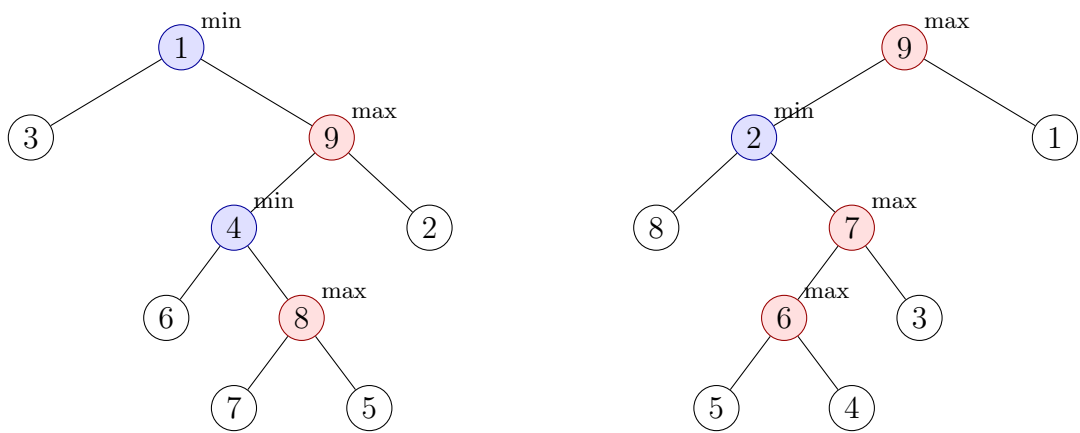
\begin{figure}[ht]
\begin{tikzpicture}[level distance=12mm,
  level 1/.style={sibling distance=40mm},
  level 2/.style={sibling distance=26mm},
  level 3/.style={sibling distance=18mm}]

\node[minnode] (r1) {1}
  child{ node[leaf] {3} }
  child{ node[maxnode] (A1) {9}
    child{ node[minnode] (B1) {4}
      child{ node[leaf] {6} }
      child{ node[maxnode] (C1) {8}
        child{ node[leaf] {7} }
        child{ node[leaf] {5} }
      }
    }
    child{ node[leaf] {2} }
  };

\node[note, above right=-1mm and -1mm of r1] {min};
\node[note, above right=-1mm and -1mm of A1] {max};
\node[note, above right=-1mm and -1mm of B1] {min};
\node[note, above right=-1mm and -1mm of C1] {max};

\node[maxnode, right=9cm of r1] (r2) {9}
  child{ node[minnode] (A2) {2}
    child{ node[leaf] {8} }
    child{ node[maxnode] (B2) {7}
      child{ node[maxnode] (C2) {6}
        child{ node[leaf] {5} }
        child{ node[leaf] {4} }
      }
      child{ node[leaf] {3} }
    }
  }
  child{ node[leaf] {1} };

\node[note, above right=-1mm and -1mm of r2] {max};
\node[note, above right=-1mm and -1mm of A2] {min};
\node[note, above right=-1mm and -1mm of B2] {max};
\node[note, above right=-1mm and -1mm of C2] {max};




\end{tikzpicture}
\caption{Left: the tree $T(1)$ is an HR--tree. 
Right: the tree $T(9)$ is not an HR--tree, since the label $8$ 
lies in the left subtree $T^{\ell}(2)$, contradicting the HR property at node $2$.}
 \end{figure}

Let $\mathcal{HR}_n$ be the set of HR--trees of order $n$, that is, the set of HR--trees labeled $1, \ldots, n$.  
For a binary tree $T$ with 
$n$ nodes $\{a_1<a_2<\cdots <a_n\}$ we define 
the word $w(T)$ by traversing $T$ in a \emph{left-first manner}.
Hetyei and Reiner~\cite{HR98} proved that $T\mapsto w(T)$ \emph{left-first order reading} bijectively maps HR--trees in $\mathcal{HR}_n$ to permutations of the symmetric group $\S_n$.

Reversely,  
starting from  a permutation $w=w_1w_2\ldots w_n$ of $[n]$, one can  recover 
the corresponding HR--tree $T_w$ on  
  $\{w_1,\ldots, w_n\}$  recursively as follows.
\begin{itemize}
\item [(1)]
If $i$ is the least integer 
such that $w_i$ is the minimum or maximum of
$\{w_1,w_2,\ldots,w_n\}$, 
then  $w_i$ is the root of $T_w$.
\item[(2)]
Define $T_{w_1\ldots w_{i-1}}$
and $T_{w_{i+1}\ldots w_n}$   to be the left and right subtrees of $w_i$.
\end{itemize}

An inner node in $T_w$ is called a \emph{min} node ~(resp. \emph{max}) if it is the minimum letter~(resp. maximum) among all its descendants. For example, the left tree in Fig.~1 is $T_w$ with $w=316478592$.
Let $T_{w}(w_i)$ denote
 the subtree (of $T_w$)  rooted at  $w_i$. 
 Let $T_{w}^l(w_i)$ (resp. $T_{w}^r(w_i)$)  denote the 
 subtree rooted at the left-child (resp. right-child) of $w_i$.

\begin{theorem}[Hetyei-Reiner] The left-first order reading   is a bijection 
between  $\mathcal{HR}_n$ and $\S_n$.
\end{theorem}

\begin{definition}[{$\pi_i$--factorization}]
Let $Y=\{y_1<y_2<\cdots<y_n\}$ be a finite totally ordered set, and let
$\pi=\pi_1\pi_2\cdots\pi_n$ be a permutation of~$Y$ written as a word.
For a fixed index $i\in[n]$, the \emph{$\pi_i$--factorization} of~$\pi$ is the sequence
\[
(w_1,\,w_2,\,\pi_i,\,w_4,\,w_5),
\]
where
\begin{enumerate}[label=(\arabic*)]
  \item the juxtaposition product $w_1w_2\pi_i w_4w_5$ is equal to~$\pi$;
  \item $w_2$ is the longest right factor of $\pi_1\pi_2\cdots\pi_{i-1}$ all of whose letters are greater than~$\pi_i$;
  \item $w_4$ is the longest left factor of $\pi_{i+1}\pi_{i+2}\cdots\pi_n$ all of whose letters are greater than~$\pi_i$.
\end{enumerate}
Note that any of $w_1,w_2,w_4,w_5$ may be the empty word~$\varepsilon$.
\end{definition}

\begin{remark}
Equivalently, the factorization
\[
\pi = w_1\,w_2\,\pi_i\,w_4\,w_5
\]
is determined by
\begin{align*}
w_2&=\text{maximal suffix of }\pi_1\pi_2\cdots\pi_{i-1}\text{ with all letters }>\pi_i,
\\
w_4&=\text{maximal prefix of }\pi_{i+1}\pi_{i+2}\cdots\pi_n\text{ with all letters }>\pi_i,
\end{align*}
and then
\[
w_1=\pi_1\cdots \pi_{i-|w_2|-1},\qquad
w_5=\pi_{i+|w_4|+1}\cdots \pi_n .
\]
This decomposition is unique.
\end{remark}

\begin{example}
Let $Y=\{1<2<3<4<5<6<7<8\}$ and
\[
\pi=6\,4\,7\,3\,8\,2\,5\,1.
\]

\smallskip
\noindent
(1) For $i=5$, we have $\pi_i=8$.  
The left part $\pi_1\pi_2\pi_3\pi_4=6\,4\,7\,3$ has no suffix with letters $>8$,
so $w_2=\varepsilon$, $w_1=6\,4\,7\,3$.
The right part $\pi_6\pi_7\pi_8=2\,5\,1$ has no prefix $>8$,
so $w_4=\varepsilon$, $w_5=2\,5\,1$.
Hence the $8$--factorization is $(w_1,w_2,\pi_i,w_4,w_5)=(6473,\varepsilon,8,\varepsilon,251)$.

\smallskip
\noindent
(2) For $i=3$, we have $\pi_i=7$.
Then $w_2=\varepsilon$, $w_1=64$, and on the right
$\pi_4\pi_5\pi_6\pi_7\pi_8=3\,8\,2\,5\,1$
has maximal prefix $w_4=8$ of letters $>7$, giving $w_5=3\,2\,5\,1$.
Thus the $7$--factorization is $(64,\varepsilon,7,8,3251)$, i.e.,
\[
64\;\underbrace{\varepsilon}_{w_2}\;7\;\underbrace{8}_{w_4}\;3\,2\,5\,1.
\]
\end{example}

\begin{remark}[Algorithmic description]
Given $i$, the $\pi_i$--factorization can be found in linear time:
\begin{enumerate}
  \item Scan left from $i-1$ while the entries are $>\pi_i$ to collect $w_2$; stop at the first $\le\pi_i$.
  \item Scan right from $i+1$ while the entries are $>\pi_i$ to collect $w_4$; stop at the first $\le\pi_i$.
  \item The remaining prefix and suffix are $w_1$ and $w_5$.
\end{enumerate}
\end{remark}

\begin{defi}[André permutation]
A permutation $\pi\in\S_n$ is an Andr\'e permutation  (of the first kind) if the following conditions hold:
\begin{itemize}
\item[(i)]
$\pi$ has no double descents, and no final descent, i.e., $\pi_{n-1}<\pi_n$.
\item[(ii)]
For $i\in\{2,\ldots,n-1\}$, let  $(w_1,w_2,\pi_i,w_4,w_5)$ be the  $\pi_i$-factorization of $\pi$. Then the maximum letter of $w_2$ is smaller than the maximum letter of $w_4$.
\end{itemize}
A tree $T_w$ is called \emph{André tree}
if $w$ is an André permutation.
\end{defi}
Let $\A_n$ be the set of Andr\'e permutations of $[n]$.
The Andr\'e permutations of length $n=2, 3, 4$ are
\[
12; \quad 123,\; 213; \quad 1234,\; 2134,\; 2314,\; 3124,\; 1324.
\]
By the definition, the last letter of an André permutation of  $[n]$ is necessarily $n$.
\begin{lem}[Hetyei and Reiner]\label{lemma-HR}
A permutation $w\in\S_n$ is an Andr\'e permutation if and only if all inner nodes of $T_w$ are $\min$-nodes.
\end{lem}

\begin{definition}[HR-action $\psi_i$ on HR--trees]
Let $w\in\S_n$ and let $T_w$ be its HR--tree. For $1\le i\le n$, denote by $v_i$ the node
labeled $w_i$ and by $T^r(v_i)$ the right subtree of $v_i$. Define the operator
$\psi_i$ acting on $T_w$ as follows.

\smallskip
\noindent
\textbf{(a) Trivial case.} If $T^r(v_i)$ is empty, set $\psi_i(T_w)=T_w$.

\smallskip
\noindent
\textbf{(b) Nontrivial case.} Let
\[
R_i \;:=\; \{\text{labels of the nodes of }T^r(v_i)\}\,\cup\,\{w_i\},
\]
viewed as a totally ordered set with the natural order on labels. Define an
order-preserving bijection
\[
\tau_i : R_i \longrightarrow R_i
\]
by
\[
\tau_i(w_i)=
\begin{cases}
\max R_i,& \text{if }v_i\text{ is a $\min$-node},\\[2pt]
\min R_i,& \text{if }v_i\text{ is a $\max$-node},
\end{cases}
\qquad
\text{and }\quad
\begin{minipage}[c]{0.45\linewidth}
    \text{$\tau_i$ is increasing on } $R_i\setminus\{w_i\}$\\
\text{ onto }$R_i\setminus\{\tau_i(w_i)\}$.
\end{minipage}
\]
Equivalently: remove $w_i$ from $R_i$, slide the remaining labels monotonically to
fill the gap, and place $w_i$ at the extremum (maximum if $v_i$ is a $\min$-node,
minimum if $v_i$ is a $\max$-node).

Relabel the nodes in the induced subtree $U_i:=\{v_i\}\cup V\!\bigl(T^r(v_i)\bigr)$
by applying $\tau_i$ to their current labels, and leave all other nodes of $T_w$
unchanged. The shape of $T_w$ is not modified. Denote the resulting labeled tree
by $\psi_i(T_w)$. 
\end{definition}

\begin{remark}
\begin{itemize}
\item The phrase “preserve the relative orders in $T^r(v_i)$” is captured by the
order-preserving map $\tau_i$: among the nodes of $T^r(v_i)$, the pairwise order
of labels is unchanged.
\item If $v_i$ is a $\min$-node (resp.\ $\max$-node), then after applying $\psi_i$ the
new label at $v_i$ is the largest (resp.\ smallest) element of $R_i$.
\item Only labels in $\{v_i\}\cup V(T^r(v_i))$ are affected; the tree shape and all
other labels remain the same.
\end{itemize}
\end{remark}

\begin{example} Let $w=562314$ with $i=2$. See Figure~\ref{fig:HR-action}.
\textbf{Before:} $v_i$ is a $\max$-node, $w_i=6$, 
$R_i=\{1,2,3,4,6\}$, right subtree labels are $\{1,2,3,4\}$, and
\(\displaystyle
\tau_i:\; 1\mapsto2,\,2\mapsto3,\,3\mapsto4,\,4\mapsto6,\, 6\mapsto1.
\)
\textbf{After:} apply $\psi_i$ (shape unchanged) 
new labels on $\{v_i\}\cup T^r(v_i)$ via $\tau_i$. 

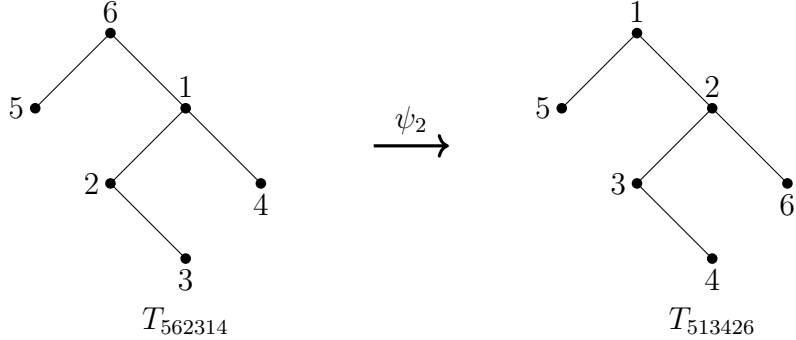
\begin{figure}[ht]
\begin{tikzpicture}
\draw (0,0)--(-1,-1);
\draw (0,0)--(1,-1)--(2,-2);
\draw (1,-1)--(0,-2)--(1,-3);
\fill (0,0) circle (2pt);
\node[above] at (0,0) {$6$};
\node[left] at (-1,-1) {$5$};
\node[above] at (1,-1) {$1$};
\node[left] at (0,-2) {$2$};
\node[below] at (1,-3) {$3$};
\node[below] at (2,-2) {$4$};
\fill (-1,-1) circle (2pt);
\fill (1,-1) circle (2pt);
\fill (2,-2) circle (2pt);
\fill (0,-2) circle (2pt);
\fill (1,-3) circle (2pt);
\node[below] at (1,-3.5) {$T_{562314}$};
\draw[][->,very thick](3.5,-1.5)--(4.5,-1.5) node[midway, above]{$\psi_2$};
\draw (7,0)--(6,-1);
\draw (7,0)--(8,-1)--(9,-2);
\draw (8,-1)--(7,-2)--(8,-3);
\fill (7,0) circle (2pt);
\node[above] at (7,0) {$1$};
\node[left] at (6,-1) {$5$};
\node[above] at (8,-1) {$2$};
\node[left] at (7,-2) {$3$};
\node[below] at (8,-3) {$4$};
\node[below] at (9,-2) {$6$};
\fill (6,-1) circle (2pt);
\fill (8,-1) circle (2pt);
\fill (9,-2) circle (2pt);
\fill (7,-2) circle (2pt);
\fill (8,-3) circle (2pt);
\node[below] at (8,-3.5) {$T_{513426}$};
\end{tikzpicture}
\caption{ HR-action $\psi_2$ on HR--tree $T_{562314}$}
\label{fig:HR-action}
\end{figure}
\end{example}

\begin{lem}[Hetyei and Reiner]
The action $\psi_i$ is an involution  on 
$\mathcal{HR}_n$; and 
$\psi_i$ and $\psi_j$ commute for all $i,j\in [n]$.
\end{lem}

 For any subset $S=\{s_1,\ldots, s_r\}$ of $[n]$ we define the operator   $\psi_S$ on $\mathcal{HR}_n$ by 
$$
\psi_S(T)=\psi_{s_r}\cdots\psi_{s_1}(T)\qquad \textrm{for}\quad T\in \mathcal{HR}_n.
$$
The group $\mathbb{Z}_2^n$ acts on $\hrn$ via the operators
 $\psi_S$, $S\subseteq[n]$, and  denote the orbit of $T$ under the 
\emph{HR--action} by  
$$
\mathrm{Orb}(T)=\{g(T):g\in\mathbb{Z}_2^n\}.
$$

For a binary tree $T$, we denote by  $\leaf(T)$  the number of leaves of $T$.
\begin{fact}\label{F-1} 
If  a binary tree $T$ has $i-1$ nodes with two children, then 
$\leaf(T)=i$.
\end{fact}




\begin{remark}
Intuitively, each time we add a node with two children, we replace one existing leaf
by two new child nodes, increasing the number of leaves by $1$. 
Starting with one leaf (the root), after adding $i-1$ such nodes we get $i$ leaves.
\end{remark}

\begin{prop}\label{p3}
Let $T\in \mathcal{HR}_n$. 
 If $\leaf(T)=i$, then $w(T)$ has at least $i$ alternating runs.
\end{prop}
\begin{proof}
By Fact~\ref{F-1}, if
$T$ has $i$ leaves, then  $T$ has $i-1$ nodes with two children. Since $T$ is a HR--tree, every node with two children in  $T$ must
 be a valley or peak in $w(T)$. According to the definition of alternating run, a permutation changes direction at each valley or peak, hence the permutation $w(T)$ has at least $i$ alternating runs.
\end{proof}
\begin{remark} It is possible that $\run(w)>\leaf(T_w)$. For example,
the permutation $w=3245617$ has two valleys at positions 2, 6 and one peak at 5, hence $\run(w)=4$, whereas $\leaf(T_w)=3$.
\end{remark}
\begin{remark} 
The HR--action is different from B\'ona's group action in \cite{Bo21}. 
For example, if $\pi = 562314$, the HR--action gives $513426$ (see Figure~2), 
while B\'ona's action of $\pi_2$ gives $514362$. 
Also, when $n=4$, B\'ona's action divides $\mathfrak{S}_4$ into two orbits, 
whereas the HR--action yields five orbits.
\end{remark}
\section{Alternating runs in Type A}\label{sec:typeA}
\subsection{Alternating subsequences}
For our approach, it is convenient to first examine the closely related problem of alternating subsequences, introduced by Stanley~\cite{St08}.
 Let $\pi=\pi_1\pi_2\ldots\pi_n\in\S_n$. A subsequence
  $\pi_{i_1}\pi_{i_2}\ldots\pi_{i_k}$ of $\pi$, where $i_1<\cdots <i_k$, is an \emph{alternating subsequence} 
  if 
  $$
  \pi_{i_1}>\pi_{i_2}<\pi_{i_3}>\cdots\pi_{i_k}.
  $$ 
Let  $\as(\pi)$ be the length of the \emph{longest alternating subsequences} of $\pi$. Let 
$a_k(n)=\#\{\pi\in\S_n:\as(\pi)=k\}$. It is easy to see that 
\begin{equation}\label{as:run}
\as(\pi)=\begin{cases}
\mathrm{run}(\pi), & \textrm{if} \;\;\pi_1<\pi_2,\\
\mathrm{run}(\pi)+1, & \textrm{if} \;\;\pi_1>\pi_2.
\end{cases}
\end{equation}

\begin{lemma}\label{L-1}
For every $\pi\in\S_n$, there exists a unique tree $T_{\check{\pi}}$ in the HR--orbit $\Orb(T_\pi)$, with $\check{\pi}\in\S_n$, such that 
\[
\as(\check{\pi})=\leaf(T_\pi).
\]
\end{lemma}

\begin{proof}
By definition, the HR–action preserves the number of inner nodes with two children of a binary tree.
Moreover,  It follows from  Proposition~\ref{p3} and \eqref{as:run} that 
\[
\as(\sigma)\geq \leaf(T_\pi)\]
for every $T_\sigma\in\Orb(T_\pi)$, where $\sigma\in \S_n$.

Let $\pi_{i_1}\pi_{i_2}\cdots\pi_{i_k}$ be the sequence of labels of the nodes with two children in $T_\pi$, listed in increasing order of their positions: $i_1<i_2<\cdots<i_k$.  
For each $1\le j\le k$, by applying HR--action if necessary, we may transform $T_\pi$ into a tree $T_{\pi'}\in\Orb(T_\pi)$, with $\pi'\in\S_n$, such that 
\[
\pi'_{i_1},\pi'_{i_2},\ldots,\pi'_{i_k}
\]
are alternatively the labels of max-nodes and min-nodes of $T_{\pi'}$.

Write
\[
\pi' \;=\; w_1\,\pi'_{i_1}\,w_2\,\pi'_{i_2}\,\cdots\, w_k\,\pi'_{i_k}\, w_{k+1},
\]
where each $w_j$ is (possibly empty) a consecutive block of letters of $\pi'$.
Now apply HR--action to $T_{\pi'}$ so that every node whose label lies in $w_j$ becomes a min-node when $j$ is odd and a max-node when $j$ is even, for $1\le j\le k+1$.  
Let $T_{\check{\pi}}$ be the resulting tree, with $\check{\pi}\in\S_n$.  
Combining the above operation and \eqref{as:run}, we obtain
\(\as(\check{\pi})\;=\;\leaf(T_\pi)\).

Finally, $T_{\check{\pi}}$ is the unique tree in $\Orb(T_\pi)$ satisfying this identity.  
Indeed, any other pattern of HR--actions on the nodes in the blocks $w_j$, $1\le j\le k+1$, would necessarily produce at least one additional alternating run in the corresponding permutation.
\end{proof}


It is clear that the permutation $\check{\pi}$ in Lemma~\ref{L-1} satisfies
\begin{equation}\label{as:lpv}
    \as(\check{\pi})
   \;=\;
   \min\{\as(\sigma):\, T_\sigma\in\Orb(T_\pi)\}=\leaf(T_\pi).
\end{equation}
In what follows we shall use $\check{\pi}$ to denote this unique permutation with 
$T_{\check{\pi}}\in\Orb(T_\pi)$.

For example, let $\pi=513426$.  
The subsequence of labels of the nodes with two children in $T_\pi$ is $\pi_2\pi_5$.  
By applying HR--action (if necessary), we may transform $T_\pi$ into a tree $T_{\pi'}$ for which 
$\pi_2'$ is a max-node and $\pi_5'$ is a min-node; the resulting permutation is $\pi'=562314$.

\begin{figure}[ht]
\begin{tikzpicture}
\draw (-4,0)--(-5,-1);
\draw (-4,0)--(-3,-1)--(-2,-2);
\draw (-3,-1)--(-4,-2)--(-3,-3);
\fill (-4,0) circle (2pt);
\node[above] at (-4,0) {$1$};
\node[left] at (-5,-1) {$5$};
\node[above] at (-3,-1) {$2$};
\node[left] at (-4,-2) {$3$};
\node[below] at (-3,-3) {$4$};
\node[below] at (-2,-2) {$6$};
\fill (-5,-1) circle (2pt);
\fill (-3,-1) circle (2pt);
\fill (-2,-2) circle (2pt);
\fill (-4,-2) circle (2pt);
\fill (-3,-3) circle (2pt);
\node[below] at (-3,-3.5) {$T_{5\red{1}34\red{2}6}$};
\draw[][->,very thick](-1.5,-1.5)--(-0.5,-1.5) node[midway, above]{$\psi_2$};
\draw (1,0)--(0,-1);
\draw (1,0)--(2,-1)--(3,-2);
\draw (2,-1)--(1,-2)--(2,-3);
\fill (1,0) circle (2pt);
\node[above] at (1,0) {$6$};
\node[left] at (2.25,-0.7) {$1$};
\node[above] at (-0.25,-1.25) {$5$};
\node[left] at (3.25,-2.25) {$4$};
\node[below] at (1,-2) {$2$};
\node[below] at (2,-3) {$3$};
\fill (1,-2) circle (2pt);
\fill (0,-1) circle (2pt);
\fill (2,-1) circle (2pt);
\fill (3,-2) circle (2pt);
\fill (2,-3) circle (2pt);
\node[below] at (2,-3.5) {$T_{5\red{6}23\red{1}4}$};
\draw[][->,very thick](3.5,-1.5)--(4.5,-1.5) node[midway, above]{$\psi_3$};
\draw (6,0)--(5,-1);
\draw (6,0)--(7,-1)--(8,-2);
\draw (7,-1)--(6,-2)--(7,-3);
\fill (6,0) circle (2pt);
\node[above] at (6,0) {$6$};
\node[left] at (7.25,-0.7) {$1$};
\node[above] at (4.75,-1.25) {$5$};
\node[left] at (8.25,-2.25) {$4$};
\node[below] at (6,-2) {$3$};
\node[below] at (7,-3) {$2$};
\fill (6,-2) circle (2pt);
\fill (5,-1) circle (2pt);
\fill (7,-1) circle (2pt);
\fill (8,-2) circle (2pt);
\fill (7,-3) circle (2pt);
\node[below] at (7,-3.5) {$T_{5\red{6}32\red{1}4}$};
\end{tikzpicture}
\caption{ A HR--action:  $\psi_3\psi_2(T_{5\red{1}34\red{2}6})=T_{5\red{6}32\red{1}4}$.}\label{MNT2}
\end{figure}

We write
\[
\pi' = w_1\,6\, w_2\, 1\, w_3,
\qquad\text{where } w_1=5,\; w_2=23,\; w_3=4.
\]
Applying HR--action to $T_{\pi'}$, we turn all nodes whose labels lie in $w_1$ and $w_3$ into min-nodes,  
and all nodes whose labels lie in $w_2$ into max-nodes.  
This yields the tree $T_{\check{\pi}}$ with
\[
\check{\pi} = 563214.
\]
See Figure~\ref{MNT2}.

\begin{defi}
Let $\pi=\pi_1\pi_2\ldots\pi_n\in\mathfrak{S}_n$.  
\begin{itemize}
\item An index $i\in\{2,3,\ldots,n-1\}$ is a \emph{peak} (resp.\ a \emph{valley}) of $\pi$ if  
\[
\pi_{i-1}<\pi_i>\pi_{i+1}
\quad
(\text{resp. }\pi_{i-1}>\pi_i<\pi_{i+1}).
\]

\item An index $i\in[n-1]$ is a \emph{left-peak} of $\pi$ with $\pi_0=0$, if  
\[
\pi_{i-1}<\pi_i>\pi_{i+1}.
\]
\end{itemize}
\end{defi}

For $\pi\in\S_n$, we denote by $\mathrm{Pk}(\pi)$, $\mathrm{Val}(\pi)$, and $\Lpk(\pi)$ the sets of peaks, valleys, and left-peaks of $\pi$, respectively.  
We also set
\begin{subequations}
\begin{equation}  
\LPV(\pi)=\Lpk(\pi)\cup\Val(\pi),
\end{equation}
the set of left-peaks and valleys of $\pi$.

Since a permutation changes direction only at a peak or a valley, we have
\begin{equation}  
\run(\pi)=|\mathrm{Pk}(\pi)|+|\mathrm{Val}(\pi)|+1.
\end{equation}
Moreover, every node with two children in $T_\pi$ corresponds to a peak or a valley of $\pi$, hence
\begin{equation}
    \leaf(T_\pi)\le |\LPV(\pi)|+1.
\end{equation}
It is also immediate from \eqref{as:run} that
\begin{equation}\label{eq:as:Lpv}
\as(\pi)=|\LPV(\pi)|+1.
\end{equation}
\end{subequations}

\begin{lemma}\label{L-2} 
Let $S\subseteq[n-1]$ and  $\pi\in\S_n$.  If $\LPV(\check{\pi})\subseteq S$, then there is a unique $T_\sigma\in\mathrm{Orb}(T_\pi)$ such that $\LPV(\sigma)=S$.
\end{lemma}

\begin{proof}
Let  $S=\{s_1<s_2<\cdots<s_k\}$. We construct a permutation $\sigma\in \S_n$ with $\LPV(\sigma)=S$ by the following process. Because $\sigma_0=0$, so $s_1$ should be a left peak in $\sigma$, since the valleys and peaks in a permutation are alternating,  using HR--action on $T_{{\pi}^\ast}$, we make $s_i$ a $\max$-node if $i$ is odd, and make $s_i$ a $\min$-node if $i$ is even. And make every node whose position is between $s_{i-1}$ and $s_i$ a min (resp. max) node if $i$ is odd (resp. even), here $s_0=0$ and {$s_{k+1}=n+1$}. Define $T_\sigma$ to be the corresponding HR--tree. It is clear that $T_\sigma$ is the only tree  in $\mathrm{Orb}(T_\pi)$ satisfying $\LPV(\sigma)=S$.
\end{proof}

\begin{lemma}\label{T-1}
For  $\pi\in\S_n$, we have
\begin{align}\label{k-1}
\sum_{T_\sigma\in\mathrm{Orb}(T_\pi)}x^{\as(\sigma)}=x^{\leaf(T_\pi)}(1+x)^{n-\leaf(T_\pi)}.
\end{align}
\end{lemma}

\begin{proof}
By the proof of Lemma~\ref{L-2}, we have for any 
$T_\sigma\in\mathrm{Orb}(T_\pi)$ that 
\[
\LPV(\check{\pi})\subseteq \LPV(\sigma).
\]
Hence, using~\eqref{eq:as:Lpv},
\[
\as(\sigma)-\as(\check{\pi})
   = \bigl|\LPV(\sigma)\setminus \LPV(\check{\pi})\bigr|.
\]
It follows that
\begin{align}\label{k-1'}
\sum_{T_\sigma\in\mathrm{Orb}(T_\pi)} x^{\as(\sigma)}
&= x^{\as(\check{\pi})}
   \sum_{T_\sigma\in\mathrm{Orb}(T_\pi)}
       x^{\left|\LPV(\sigma)\setminus \LPV(\check{\pi})\right| }\nonumber\\
&= x^{\as(\check{\pi})}
   \sum_{S\subseteq [n-1]\setminus \LPV(\check{\pi})}
      x^{|S|}. 
\end{align}

By Lemma~\ref{L-1}, we have 
$\as(\check{\pi})=\leaf(T_\pi)$ and 
$|\LPV(\check{\pi})|=\leaf(T_\pi)-1$.  
Thus,
\[
\bigl|[n-1]\setminus \LPV(\check{\pi})\bigr|
  = (n-1)-(\leaf(T_\pi)-1)
  = n-\leaf(T_\pi).
\]
Substituting this into~\eqref{k-1'} and applying the binomial
formula yields~\eqref{k-1}.
\end{proof}


\begin{theorem}\label{C-1}
We have
\begin{align}\label{as-enumerator}
\sum_{\sigma\in\S_n}x^{\mathrm{as}(\sigma)}=\sum_{i=0}^{\lfloor(n-2)/{2}\rfloor}d_{n,i}x^{i+1}(1+x)^{n-i-1},
\end{align}
where $d_{n,i}$ is the number of  Andr\'e  $n$-permutations with  $i$ descents.
\end{theorem}
\begin{proof}
For $\pi\in\mathfrak{S}_n$,  
the orbite $\mathrm{Orb}(T_\pi)$ contains a
 unique  tree $T_{\tilde{\pi}}$, of which  all the inner nodes are $\min$-nodes. By Lemma~\ref{lemma-HR},
$\tilde{\pi}$ is an Andr\'e permutation.
Consequently, $|\mathrm{Lpk}(\tilde{\pi})|=\mathrm{val}(\tilde{\pi})$, since $\tilde{\pi}_n=n$.
Moreover, because every inner node of $T_{\tilde{\pi}}$ is a $\min$--node,
we obtain
\[
\mathrm{val}(\tilde{\pi})+1=\leaf(T_\pi).
\]
Finally, by Lemma~\ref{T-1}, the desired result follows.
\end{proof}
\begin{table}[t]
  $$
  \begin{array}{c|ccccc}
\hbox{$n$}\backslash\hbox{$k$}&0&1&2&3\\
\hline
1& 1&\\
2& 1&\\
3& 1&2&&\\
4& 1&8&&\\
5& 1&22&16&\\
6& 1&52&136&\\
7& 1&114&720&272\\
\end{array}
\qquad\qquad\qquad
 \begin{array}{c|cccc|c}
 n\diagdown  k &0&1&2&3&E_n=\sum_k d_{n,k}\\
 \hline
 1&1&&&&1\\
 2&1&&&&1\\
 3&1&1&&&2\\
 4&1&4&&&5\\
 5&1&11&4&&16\\
 6&1&26&34&&61\\
 7&1&57&180&34&272\\
 \end{array}
 $$
 \vspace{10pt}
 \caption{The first values of $2^k d_{n,k}$ (left), $d_{n,k}$ and  $E_n$ for $0\leq 2k< n\leq 7$. \label{table-1}}
 \end{table}
 
 The Euler numbers $E_n$ are  defined by the exponential generating function
\[
\sum_{n\geq 0} E_n\frac{x^n}{n!}=\tan (x) +\sec(x).
\]
The array $(d_{n,k})_{n,k\ge0}$ corresponds to
\href{https://oeis.org/A094503}{OEIS~A094503}~\cite{oeis22}.
It satisfies the recurrence relation
\[
d_{n,k} \;=\; (k+1)\,d_{n-1,k} + (n-2k)\,d_{n-1,k-1},
\]
together with the summation identity
\[
\sum_k d_{n,k} \;=\; E_n.
\]
See Table~\ref{table-1}.
\subsection{Alternating runs }

Define the \emph{André polynomials}~\cite{FH01} by
\[
D_n(x)=\sum_{i=0}^{\lfloor (n-1)/2 \rfloor} d_{n,i}\, x^{\,i+1},
\]
where $d_{n,i}$ denotes the number of André $n$–permutations with $i$ descents.

\begin{theorem}
For $n\ge 2$, we have
\begin{subequations}
    \begin{equation}\label{R-D}
    R_n(x)
    = 2(1+x)^{\,n-1}\,
      D_n\!\left(\frac{x}{1+x}\right).
\end{equation}
Equivalently,
\begin{align}\label{eq-2}
R_n(x)
   = 2 \sum_{i=0}^{\lfloor (n-1)/2 \rfloor}
       d_{n,i}\, x^{\,i+1}(1+x)^{\,n-2-i}.
\end{align}
\end{subequations}
\end{theorem}

\begin{proof}
Define the two subsets of $\S_n$:
\begin{align*}
\S_n^> &:= \{\sigma\in\S_n : \sigma_1>\sigma_2\},\\
\S_n^< &:= \{\sigma\in\S_n : \sigma_1<\sigma_2\}.
\end{align*}
The \emph{complementation} $\sigma\mapsto\sigma^{\mathrm{c}}$, defined by  
$\sigma^{\mathrm{c}}_i = n+1-\sigma_i$ for $1\le i\le n$,  
is a bijection between these two sets and preserves the number of alternating runs; that is,
\[
\run(\sigma)=\run(\sigma^{\mathrm{c}}).
\]
Therefore, by \eqref{def: run polynomials},
\begin{align}\label{eq-1}
R_n(x)=2\sum_{\sigma\in\S_n^>} x^{\run(\sigma)}=2\sum_{\sigma\in\S_n^<} x^{\run(\sigma)}, \qquad n\ge 2.
\end{align}
Combining \eqref{as:run} with \eqref{eq-1}, we obtain
\begin{align*}
2\sum_{\sigma\in\S_n} x^{\as(\sigma)}
   = x 2\sum_{\sigma\in\S_n^{>}} x^{\run(\sigma)}
      \;+\; 2\sum_{\sigma\in\S_n^{<}} x^{\run(\sigma)} =(1+x) R_n(x).
\end{align*}
The result then follows from Theorem~\ref{C-1}.
\end{proof}

The first few instances of \eqref{eq-2} are as follows:
\begin{align*}
R_1(x) &= x, 
\qquad R_2(x) = 2x,\\[2mm]
R_3(x) &= 2x(1+x) + 2x^{2},\\[1mm]
R_4(x) &= 2x(1+x)^{2} + 8x^{2}(1+x),\\[1mm]
R_5(x) &= 2x(1+x)^{3} + 22x^{2}(1+x)^{2} + 8x^{3}(1+x),\\[1mm]
R_6(x) &= 2x(1+x)^{4} + 52x^{2}(1+x)^{3} + 68x^{3}(1+x)^{2}.
\end{align*}

\begin{remark}
Ma~\cite[Theorem~11]{Ma13} derived~\eqref{R-D} from the known
difference–differential equations satisfied respectively by 
$R_n(x)$ and $D_n(x)$.

\end{remark}
 
\begin{coro}   
When $n\geq 2$ we have 
\begin{align}\label{factor}
    R_n(x)=2(1+x)^{\lfloor (n-2)/2 \rfloor}\sum_{k=1}^{\lfloor (n+1)/2\rfloor}d_{n,k-1}x^k(1+x)^{\lfloor (n+1)/2\rfloor-k}.
\end{align}
\end{coro}

Hence, for $n\geq 2$, the polynomial $R_n(x)$ is divisible by 
$(1+x)^{\lfloor (n-2)/2 \rfloor}$  and the quotient is given by
\begin{equation}\label{def:M}
    M_n(x):=2\sum_{k=1}^{\lfloor (n+1)/2\rfloor}d_{n,k-1}x^k(1+x)^{\lfloor (n+1)/2\rfloor-k}.
\end{equation}

The first few polynomials of $M_n(x)$ with $n\geq 3$  are as follows:
\begin{align*}
M_3(x)&=2(x+2x^2),\\
    M_4(x)&=2(x+5x^2),\\
     M_5(x)&=2(x+3x^2+4x^3),\\
     M_6(x)&=2(x+28x^2+61x^3).
\end{align*}
\begin{remark} 
Since  $M_n(-1)=2 \,(-1)^{\lfloor (n+1)/2\rfloor} d_{n,\lfloor (n+1)/2\rfloor-1}\neq 0$, we deduce that 
$-1$ is a zero  of $R_n(x)$ of order $\lfloor (n-2)/2 \rfloor$. This result is stronger than merely asserting that
\(R_n(x)\) has a factor \((x+1)^{\lfloor (n-2)/2 \rfloor}\), see \eqref{multiplicity-m}.
\end{remark}

By \eqref{def:M} it is clear that $M_n(x)$ is a polynomial with nonnegative integer coefficients.
B\'ona~\cite{Bo21} provided a combinatorial interpretation of the polynomial~$M_n(x)$.
In what follows, we give an alternative interpretation of $M_n(x)$
in terms of HR--trees.

\begin{definition}
    Let $T_\pi\in \mathcal{HR}_n$ and  let \[
\pi_{i_1},\,\pi_{i_2},\,\cdots,\pi_{i_m} \; (i_1<i_2\cdots<i_m)\]
    be the labels of the nodes of $T_\pi$ having exactly one child.
    The node labeled $\pi_{i_k}$ 
    is called the $k$-th node with one child. It 
is said to be an even (resp. odd) node with one child if $k$ is even (resp. odd). We denote by $\widetilde{\hr}_n$ 
the set of the trees $T\in \mathcal{HR}_n$  satisfying the following two conditions:
\begin{itemize}
\item all nodes with two children of $T$ are $\min$-nodes;
\item all $\min$-nodes with  one child of $T$ are  even  nodes.
\end{itemize}
\end{definition}

\begin{example}
Figure~\ref{fig:hr-example} illustrates a tree $T_\pi\in\widetilde{\mathcal{HR}}_5$ with $\pi=42135$.
In this example, the node with two children is the node $1$, and
the $\min$-node with one child is labeled with  $\pi_4=3$,
so $T_\pi$ indeed belongs to~$\widetilde{\mathcal{HR}}_5$.
\end{example}

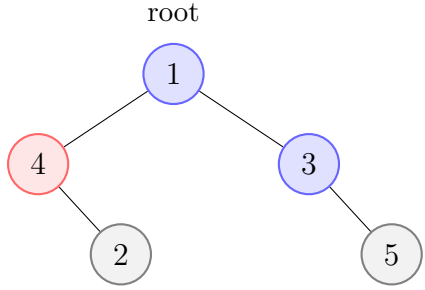
\begin{figure}[ht]
\centering
\begin{tikzpicture}[
  level distance=12mm,
  every node/.style={circle,draw,minimum size=8mm,inner sep=0pt,thick},
  minnode/.style={draw=blue!60,fill=blue!12},
  maxnode/.style={draw=red!60,fill=red!10},
  leafnode/.style={draw=black!50,fill=black!5},
  evenring/.style={draw=green!70!black, line width=1.2pt},
  level 1/.style={sibling distance=36mm},
  level 2/.style={sibling distance=22mm}
]

\node[minnode,label=above:{\small root}] (r){$1$}
  child { node[maxnode] (a){$4$}
      child[missing]
    child { node[leafnode] (b){$2$} }
  }
  child { node[minnode] (b){$3$}
    child[missing]
    child { node[leafnode] (bR){$5$} }
  };
\end{tikzpicture}
\caption{A tree $T_\pi\in\widetilde{\mathcal{HR}}_5$. }
\label{fig:hr-example}
\end{figure}

\begin{theorem}
    For $n\geq 2$, we have
\begin{equation}\label{3.9}
    M_n(x)=2\sum_{T\in \widetilde{\hr}_n}x^{\min(T)+1},
\end{equation}
 where $\min(T)$ is the number of $\min$-nodes of T.
\end{theorem}
\begin{proof}  Let $\ell=\lfloor (n+1)/2\rfloor$ and 
\begin{equation}\label{def:M-p}
M_n(x)=2\sum_{i=1}^{\ell}q_{n,i}x^{\ell-i}.
\end{equation}
It follows from \eqref{def:M} and \eqref{def:M-p} that 
\begin{equation}
    q_{n,i}=\sum_{k=1}^{\ell-i}
\binom{\ell-k}{i}
 d_{n,k-1} \quad (i=0, \ldots, \ell-1).
\end{equation}
For $i\in \{0, \ldots, \ell-1\}$, let $\widetilde{\hr}_{n,i}$ denote the set of trees $T$ in $\widetilde{\hr}_n$ such that 
\begin{equation}
\min(T)=\ell-i-1.
\end{equation}
Thus,  Eq~\eqref{3.9} is equivalent to 
  \begin{align}\label{eq:key}
|\widetilde{\hr}_{n,i}|=\sum_{k=1}^{\ell-i}
\binom{\ell-k}{i}
 d_{n,k-1} \quad (i=0, \ldots, \ell-1).
\end{align}
Let  $\pi$ be an André $n$-permutation with $k-1$ descents. 
By Lemma~\ref{lemma-HR}, all the inner nodes of the tree $T_{\pi}\in \mathcal{HR}_{n}$ are $\min$-nodes, of which 
$k-1$ nodes have  two children and  $n - 2k + 1$ nodes have one child.
Hence there are $\lfloor (n - 2k + 1)/2 \rfloor = \ell - k$ even $\min$-nodes with one child.
We now map $T_{\pi}$ to a tree $\widetilde{T}_\pi\in \widetilde{\mathcal{HR}}_{n,i}$ with $i\in \{0, 1, \ldots, \ell-k\}$
as follows.

Among the $\ell-k$ even nodes with one child, choose $\ell-k-i$ of them (there are
$\binom{\ell-k}{i}$ such  choices), and then apply  the HR–action to turn every remaining
one–child node—odd or even—into a $\max$–node. In the resulting tree
$\widetilde T_\pi$, the $\min$–nodes are precisely the original $k-1$
$\min$–nodes with two-children together with the $\ell-k-i$ even one–child nodes that were
not flipped; hence their number is
\[
(k-1)+(\ell-k-i)=\ell-i-1.
\]
Therefore $\widetilde T_\pi\in\widetilde{\mathcal{HR}}_{n,i}$. Since there are
$d_{n,k-1}$ André $n$-permutations with $k-1$ descents, equation
\eqref{eq:key} follows.

\end{proof}

\subsection{Link to Eulerian polynomials}\label{link to eulerian polynomials}
Define the \emph{Eulerian polynomial} \cite[Sec.~1.4]{St97} by
\[
A_n(x)
   = \sum_{\sigma\in \mathfrak{S}_n}
      x^{\,\des(\sigma)+1}\,,
\]
where $\des(\sigma)$ denotes the number of descents of $\sigma$.

A polynomial $h(t)=\sum_{i=0}^n h_i t^{i}$ with real coefficients is said to be \emph{palindromic} if  
$h_i = h_{n-i}$ for all $0\le i\le n/2$.  
It is well known that any palindromic polynomial admits the following so-called 
$\gamma$–decomposition:
\[
h(t)=\sum_{i=0}^{\lfloor n/2 \rfloor} \gamma_i\, t^{i}(1+t)^{\,n-2i}.
\]
If all the coefficients $\gamma_i$ are nonnegative, then $h(t)$ is said to be \emph{$\gamma$–nonnegative}.

The following $\gamma$–positive expansion of the Eulerian polynomial $A_n(x)$ is well known; 
see \cite{At18, Pe15, FS71}, \cite[Section~4]{FS73}, and also the $(p,q)$–analogue of \eqref{eq:FS} in \cite{PZ21}.  
For completeness, we provide a combinatorial proof using HR--action.

\begin{theorem}[Foata--Schützenberger]\label{thm:FS}
There exist positive integers $d_{n,k}$ such that
\begin{align}\label{eq:FS}
A_{n}(x)
   = \sum_{k=0}^{\lfloor (n-1)/2 \rfloor}
       2^{k}\, d_{n,k}\, x^{\,k+1}(1+x)^{\,n-1-2k},
\end{align}
where $d_{n,k}$ denotes the number of André $n$–permutations with $k$ descents.
\end{theorem}

\begin{proof}[Proof of Theorem~\ref{thm:FS}]
Let $\pi\in\A_n$, and $T_\pi$ is the corresponding HR--tree, we have the following observations:
\begin{itemize}
    \item 
    If $i$ is a valley of $\pi$, i.e., $\pi_{i-1}>\pi_i<\pi_{i+1}$ for $i\in[n-1]$, then the node labeled by $\pi_i$ of $T_\pi$ is an inner node with two children;
    \item 
    If $\pi_i$ is a label of an inner node with two children of $T_\pi$, then $i-1$ is a descent position, $i$ is a descent position of $w(\psi_i(\pi))$, and $\des\,\psi_i(\pi)=\des\,\pi$;
    \item 
    If $\pi_i$ is a label of an inner node with one child of $T_\pi$, then $i$ is an ascent position of $\pi$, $i$ is a descent position of $w(\psi_i(\pi))$, and $\des\,\psi_i(\pi)=\des\,\pi+1$.
\end{itemize}
Let
\begin{align}
c_i(\pi)=
\begin{cases}
 2x,& \text{if $\pi_i$ is a label of an inner node with two children in $T_\pi$};\nonumber\\
 1+x,&\text{if $\pi_i$ is a label of an inner node with one child in $T_\pi$};\\
 1,&\text{if $\pi_i$ is a label of a leaf in $T_\pi$.}\\
 \end{cases}
 \end{align} 
 If $T_\pi$ has  
  $k$  inner nodes with 2 children, then  $\leaf(T_\pi)=k+1$, and $T_\pi$ has $n-2k-1$ inner nodes with one child.  Applying the above observations we have
\begin{align*}
\sum_{T\in\mathrm{Orb}(T_\pi)}x^{\mathrm{des}\,w(T)}&=c_1(\pi)c_2(\pi)\cdots c_n(\pi)\\
&=(2x)^k(1+x)^{n-2k-1}.
\end{align*}
Summing over all $\pi\in\A_n$ we  obtain the result.
\end{proof}

From  \eqref{eq-2} and \eqref{eq:FS} 
we derive the following formula due to David and 
 Barton~\cite[pp. 157-162]{DB62},
 see also Knuth~\cite[p. 605]{Kn98} and Stanley~\cite[p. 685]{St08}.

 \begin{proposition}[David-Barton]    For $n\geq 2$ and $w=\sqrt{\frac{1-x}{1+x}}$,
 \begin{align}\label{Run-Euler}
R_n(x)=\left(\frac{1+x}{2}\right)^{n-1} (1+w)^{n+1}
A_n\left(\frac{1-w}{1+w}\right).
\end{align}
 \end{proposition}

Combining the combinatorial proofs of \eqref{eq-2} and \eqref{eq:FS}, 
we thereby obtain a \emph{purely combinatorial proof} of \eqref{Run-Euler}.

\section{Alternating runs  in type B}\label{sec:typeB}
\subsection{Definitions}
Let \(\B_n\) denote the hyperoctahedral group of signed permutations of \([n]\), that is,   the set of bijections  \(\pi\) of 
\([\pm n]:=[n]\cup\{-1,-2,\ldots,-n\}\) satisfying 
\(\pi(-i)=-\pi(i)\) for every \(i\in [n]\). 
Equivalently, each \(\pi\in\B_n\) can be written as a word
\(\pi_1\pi_2\cdots\pi_n\) where
\(|\pi_1|,\ldots,|\pi_n|\) form a permutation of \([n]\).

For \(\pi\in\B_n\) and \(1\le i\le n\), we denote \(\pi(i)\) simply by \(\pi_i\).
For \(1\le k\le n\), we also write \(-k\) as \(\bar{k}\).
For a signed permutation \(\pi=\pi_1\pi_2\cdots\pi_n\in\B_n\), define
\[
\Neg(\pi)
:=\{\pi_i\in[\pm n]\mid \pi_i<0\},
\]
the set of entries of \(\pi\) that appear with a negative sign.

 For \(\pi\in\B_n\), set \(\pi_0=0\).  
We say that the word \(0\pi=\pi_0\pi_1\cdots\pi_n\) \emph{changes direction} at an index
\(i\) (\(1\le i\le n-1\)) if either \(\pi_{i-1}<\pi_i>\pi_{i+1}\) or
\(\pi_{i-1}>\pi_i<\pi_{i+1}\).
We say that \(\pi\) has \(k\) \emph{alternating runs}, and write \(\run_B(\pi)=k\), if
there are exactly \(k-1\) indices at which \(0\pi\) changes direction.

For example, the signed permutation
\[
0514\bar{3}\bar{6}2
\]
has \(5\) alternating runs:
\[
05,\quad 51,\quad 14,\quad 4\bar{3}\bar{6},\quad \bar{6}2.
\]


Define the following polynomials:
\begin{subequations}
\begin{align}
R_n^{B,>}(x)&=\sum_{\pi\in\B_n^{>}} x^{\run_B(\pi)},\\
R_n^{B,<}(x)&=\sum_{\pi\in\B_n^{<}} x^{\run_B(\pi)},\\
R_n^B(x)&=\sum_{\pi\in\B_n} x^{\run_B(\pi)},
\end{align}
\end{subequations}
where
\[
\B_n^{>}=\{\pi\in\B_n:\ \pi_1>0\},
\qquad
\B_n^{<}=\{\pi\in\B_n:\ \pi_1<0\}.
\]


The first few polynomials in these three families are listed in
Table~\ref{table-2}.

\begin{table}[t]
\begin{center}
\begin{tabular}{c|c|c|c}
&\(R_n^B(x)\)&\(R_n^{B,>}(x)\)&\(R_n^{B,<}(x)\)\\
\hline
\(n=2\) & \(2x+6x^2\) & \(x+3x^2\) & \(x+3x^2\)\\
\(n=3\) & \(2x+24x^2+22x^3\) & \(x+12x^2+11x^3\) & \(x+12x^2+11x^3\)\\
\(n=4\) & \(2x+78x^2+190x^3+114x^4\) &
         \(x+39x^2+95x^3+57x^4\) &
         \(x+39x^2+95x^3+57x^4\)\\
\hline
\end{tabular}
\vspace{8pt}
\caption{First terms of \(R_n^B(x)\), \(R_n^{B,>}(x)\), and \(R_n^{B,<}(x)\).}
\label{table-2}
\end{center}
\end{table}

Let \(S=\{0,s_1,\ldots,s_n\}\) be a subset of
\(\{0,\,\pm1,\,\pm2,\,\ldots,\,\pm n\}\), where each \(s_i\) is either \(i\) or
\(\bar{i}\) for \(i\in[n]\).
Let \(\bhr_n\) denote the set of HR--trees whose vertex labels are taken from
\(S\), with the additional requirement that the leftmost node is labeled \(0\).
The HR--trees corresponding to the elements of \(\B_2\) are illustrated in
Figure~\ref{f-1}.

\begin{figure}[ht]
\begin{tikzpicture}
\draw (0,0)--(1,1);
\draw (1,1)--(2,0);
\fill (0,0) circle (2pt);
\fill (1,1) circle (2pt);
\fill (2,0) circle (2pt);
\node[below] at (0,0) {$0$};
\node[below] at (2,0) {$\bar{2}$};
\node[above] at (1,1) {$1$};
\node[below] at (1,-1) {$T_{01\bar{2}}$};
\draw (4,0)--(5,1)--(6,0);
\fill (4,0) circle (2pt);
\fill (5,1) circle (2pt);
\fill (6,0) circle (2pt);
\node[below] at (4,0) {$0$};
\node[above] at (5,1) {$2$};
\node[below] at (6,0) {$\bar{1}$};
\node[below] at (5,-1) {$T_{02\bar{1}}$};
\draw (8,1)--(8.5,0.5)--(9,0);
\fill (8,1) circle (2pt);
\fill (8.5,0.5) circle (2pt);
\fill (9,0) circle (2pt);
\node[above] at (8,1) {$0$};
\node[right] at (8.5,0.5) {$2$};
\node[below] at (9,0) {$1$};
\node[below] at (8.5,-1) {$T_{021}$};
\draw (11,1)--(11.5,0.5)--(12,0);
\fill (11,1) circle (2pt);
\fill(11.5,0.5) circle (2pt);
\fill(12,0) circle (2pt);
\node[above] at (11,1) {$0$};
\node[right] at (11.5,0.5) {$1$};
\node[below] at (12,0) {$2$};
\node[below] at (11.5,-1) {$T_{012}$};
\draw (0,-4)--(1,-3)--(2,-4);
\fill (0,-4) circle (2pt);
\fill (1,-3) circle (2pt);
\fill (2,-4) circle (2pt);
\node[below] at (0,-4) {$0$};
\node[above] at (1,-3) {$\bar{1}$};
\node[below] at (2,-4) {$2$};
\node[below] at (1,-5) {$T_{0\bar{1}2}$};
\draw (4,-4)--(5,-3)--(6,-4);
\fill (4,-4) circle (2pt);
\fill (5,-3) circle (2pt);
\fill (6,-4) circle (2pt);
\node[below] at (4,-4) {$0$};
\node[above] at (5,-3) {$\bar{2}$};
\node[below] at (6,-4) {$1$};
\node[below] at (5,-5) {$T_{0\bar{2}1}$};
\draw (8,-3)--(8.5,-3.5)--(9,-4);
\fill (8,-3) circle (2pt);
\fill (8.5,-3.5) circle (2pt);
\fill (9,-4) circle (2pt);
\node[above] at (8,-3) {$0$};
\node[right] at (8.5,-3.5) {$\bar{1}$};
\node[below] at (9,-4) {$\bar{2}$};
\node[below] at (8.5,-5) {$T_{0\bar{1}\bar{2}}$};
\draw (11,-3)--(11.5,-3.5)--(12,-4);
\fill (11,-3) circle (2pt);
\fill(11.5,-3.5) circle (2pt);
\fill(12,-4) circle (2pt);
\node[above] at (11,-3) {$0$};
\node[right] at (11.5,-3.5) {$\bar{2}$};
\node[below] at (12,-4) {$\bar{1}$};
\node[below] at (11.5,-5) {$T_{0\bar{2}\bar{1}}$};
\end{tikzpicture}
\caption{The HR--trees corresponding  to the permutations in $\B_2$ }\label{f-1}
\end{figure}
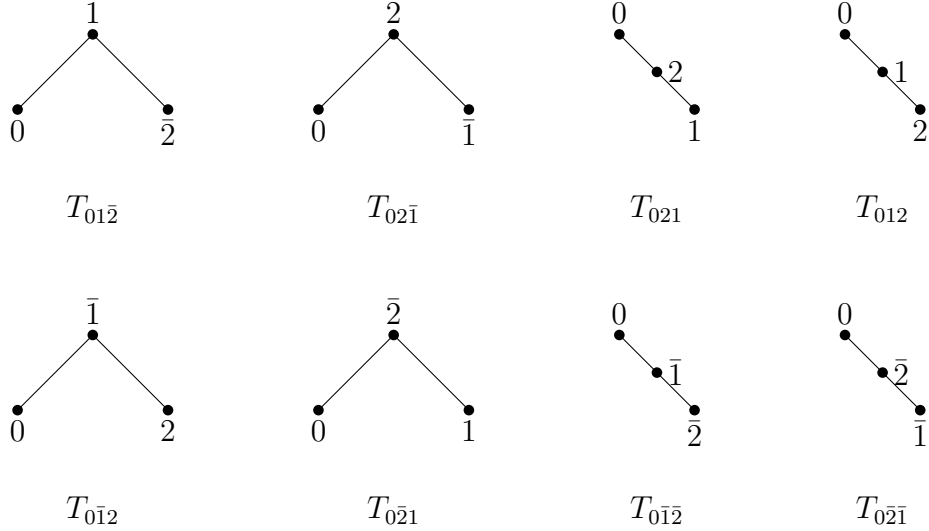

Similar to the type~A case, one readily checks that the following holds.

\begin{theorem}
The left-first order reading defines a bijection between \(\bhr_n\) and \(\B_n\).
\end{theorem}

Throughout the remainder of the paper, we adopt the convention that
\(\pi_0=0\) for every \(\pi\in\B_n\).

\begin{definition}[Type B André permutation]\label{B-André}
A signed permutation \(\pi=\pi_1\pi_2\cdots\pi_n\in\B_n\) is called a
\emph{type~B André permutation} if the following two conditions hold:
\begin{itemize}
\item[(i)] \(\pi\) has no double descents, and no final descent; that is,
\(\pi_{n-1}<\pi_n\).

\item[(ii)] For every valley \(i\in\{1,\ldots,n-1\}\) of \(\pi\), consider the
\(\pi_i\)-factorization \((w_1,w_2,\pi_i,w_4,w_5)\) of \(\pi\).  
Then the maximum letter of \(w_2\) is smaller than the maximum letter of \(w_4\).
\end{itemize}

Let \(\A_n^{B>}\) (resp. \(\A_n^{B<}\)) denote the subset of type~B André
permutations \(\pi\in\B_n\) with \(\pi_1>0\) (resp. \(\pi_1<0\)), and set
\[
\A_n^B=\A_n^{B>}\cup \A_n^{B<}.
\]
\end{definition}

For instance,
\[
\A_2^B=\{\,012,\;0\bar{1}2,\;0\bar{2}1\,\}.
\]

\begin{remark}
Alternative definitions of signed André permutations can be found in
\cite{SZ20, EF24}.
\end{remark}

\subsection{Modified type B HR--action}\label{subsection-MHR}
For $\pi=\pi_1\ldots \pi_n\in\B_n$, define
\begin{subequations}
    \begin{equation}\label{def:bar pi}
    \overline{\pi}
    = \overline{\pi}_1 \ldots \overline{\pi}_n,
\end{equation}
where the bar indicates changing the sign of each entry.  
For any subset $S\subseteq\B_n$, define
\begin{equation}\label{def:bar set}
    \overline{S}
    = \{\overline{\pi} : \pi\in S\}.
\end{equation}
\end{subequations}
\begin{lem}\label{andre-B}
 A permutation $\pi\in\B_n$ is a type B André permutation   if and only if all inner nodes of $T_\pi$ are $\min$-nodes.
 \end{lem}
\begin{proof}
Suppose that $\pi$ is a type B André permutation. Take $x$ to be the largest among all $\max$-nodes of the tree $T_\pi$, let $uxv$ be the subword of $\pi$ obtained by reading in left-first order the subtree that has $x$ as its root, and $u,v$ are subwords of $\pi$. Since the inner nodes of $\min-\max$ tree always have right subtree, so $v$ is not an empty word. Take the first letter $\mathsf{first}(v)$ of $v$ and consider  $\mathsf{first}(v)$-factorization $(w_1,w_2,\mathsf{first}(v),w_4,w_5)$ of $\pi$, since $x\in w_2$, so $w_4$ is not empty, otherwise $\mathsf{first}(v)$ is a double descent or $n-1$ is a descent which yields that $\pi$ is not an André permutation, so $\mathsf{first}(v)$ is a valley of $\pi$ and $x\leq\max(w_2)<\max(w_4)$ this is a contradiction with $x$ is the largest among all $\max$-nodes of the tree $T_\pi$.
\par
Suppose that all the inner nodes of $T_\pi$ are $\min$-nodes. So the valleys of $\pi$ must be the inner nodes of $T_\pi$ with two children. Thus for each valley $x$ of $\pi$, with $x$-factorization $(w_1,w_2,x,w_4,w_5)$, actually, the subword $w_2xw_4$ is obtained by reading the subtree whose root is $x$ in left-first order. We must have $\max(w_2)<\max(w_4)$ for otherwise $\max(w_2)$ would have been the root of this subtree. If $x$ is an inner nodes with one child, since all the inner nodes are $\min$-nodes, so $x$ either is the first letter of $\pi$ or it is a double ascent. If $x$ is a leaf of $\pi$, also by all the non zero nodes are $\min$-nodes, so $x$ is either a peak of $\pi$ or the last letter of $\pi$, so the $n-1$ position can not be a descent. So by Definition \ref{B-André}, $\pi$ is a type B André permutation. 
\end{proof}
\begin{fact}\label{fact1}
If $\pi\in\mathcal{A}_n^B$ has $k$ valleys, then $T_\pi$ has $k$ inner nodes with $2$ children.
\end{fact}
 
\begin{defi}[Modified HR--action]\label{Def-1}
Let $T_\pi\in\bhr_n$. For $1\leq i\leq n$, the operator $\widehat{\psi}_i$ acting on the tree $T_\pi$ is defined as follows.
\begin{enumerate}
\item[(a)]
For $0$ is an inner node of $T_\pi$; or $0$ is a leaf of $T_\pi$ and $\pi_i$ is not the parent of $0$.
\begin{itemize}
\item 
If $\pi_i$ is a $\min$-node, then replace $\pi_i$ by the largest element of $T_{\pi}^r(\pi_i)$, permute the remaining elements of $T_{\pi}^r(\pi_i)$ such that they keep their same relative orders and all other nodes in $T_\pi$ are fixed.
\item
If $\pi_i$ is a $\max$-node, then replace $\pi_i$ by the smallest element of $T_{\pi}^r(\pi_i)$ such that they keep their same relative order, and all other nodes in $T_\pi$ are fixed.
\end{itemize}
\vskip 1.5 mm
\item[(b)]
For $0$ is a leaf of $T_\pi$ and $\pi_i$ is the parent of $0$.
\begin{itemize}
\item
$\widehat{\psi}_i(T_\pi)=T_\pi$.
\end{itemize}
\end{enumerate}
\end{defi}
It is clear that each \(\widehat{\psi}_i\) is an involution acting on
\(\bhr_n\), and that \(\widehat{\psi}_i\) and \(\widehat{\psi}_j\) commute for
all \(i,j\in[n]\).  
Hence, for any subset \(S\subseteq[n]\), we may define the map
\[
\widehat{\psi}_S:\bhr_n\longrightarrow\bhr_n,\qquad
\widehat{\psi}_S(T_\pi)=\prod_{i\in S}\widehat{\psi}_i(T_\pi).
\]
Thus the group \(\mathbb{Z}_2^{\,n}\) acts on \(\bhr_n\) via the maps
\(\widehat{\psi}_S\) for \(S\subseteq[n]\).

For \(\pi\in\B_n\), let
\[
\mathrm{Orb}^*(T_\pi)=\{\,g(T_\pi): g\in \mathbb{Z}_2^{\,n}\,\}
\]
denote the orbit of \(T_\pi\) under the \emph{modified HR--action} (MHR--action).

\begin{proposition}
For $\pi\in\B_n^>$ (resp. $\pi\in\B_n^<$), we have,
for any $T\in\mathrm{Orb}^*(T_\pi)$, $w(T)\in\B_n^>$ (resp. $w(T)\in\B_n^<$).
\end{proposition}
\begin{proof}
For $T\in\mathrm{Orb}^*(T_\pi)$, and $\pi\in\B_n^>$.
If $0$ is a leaf of $T_\pi$ then it must be a left child of $\pi_1$, since $0$ precedes $\pi_1$ in $\pi$. And the subtree $T_{\pi,1}$ is not a part of right subtree of any nodes with two children in $T_\pi$, otherwise this would contradict the $0$ is in the first position of $\pi$. By the definition of MHR action (see, Definition \ref{Def-1}), any $\widehat{\psi}_i$ can not change the node $\pi_1$ in $T_\pi$, that is, for any $T\in\mathrm{Orb}^*(T_\pi)$, $w(T)_1=\pi_1$, i.e., $w(T)\in\B_n^>$.

If $0$ is an inner node of $T_\pi$, then $0$ has only right subtree otherwise this would contradict the $0$ is in the first position of $\pi$. Since $\pi_1>0$ and $\pi_1$ is in $T_{\pi}^r(\pi_0)$ then all the nodes of $T_{\pi}^r(\pi_0)$ are positive, if there are any nodes are negative in $T_{\pi}^r(\pi_0)$, then $0$ is neither the $\max$-nodes nor $\min$-nodes of $T_{\pi}^r(\pi_0)$, this is a contradiction. So, if $\pi_i$ is in $T_{\pi}^r(\pi_0)$, then $\widehat{\psi}_i$ does not change the positive sign of the right child of $0$ which means $w(\widehat{\psi}_i(T_\pi))_1>0$. And if $T_{\pi}(\pi_0)$ is a part of subtree of some node, then this subtree must be a left subtree, otherwise this would contradict the $0$ is in the first position of $\pi$, this shows that for any $\pi_i$ not in $T_{\pi}^r(\pi_0)$, the $\widehat{\psi}_i$ would not change any nodes in $T_{\pi}^r(\pi_0)$.

Summarize the above, for any $T\in\mathrm{Orb}^*(T_\pi)$, we have $w(T)\in\B_n^>$. Along the same lines, one can obtain for $\pi\in\B_n^<$ and for any $T\in\mathrm{Orb}^*(T_\pi)$, we have $w(T)\in\B_n^<$.
\end{proof}
For type B permutations, we also have similar Proposition as type A case, and the proof of the result is along the same lines as Proposition \ref{p3}.
\begin{proposition}\label{p4}
For a permutation $\pi\in\B_n$, if $T_\pi$ has $i$ leaves, then $\pi$ has at least $i$ alternating runs.
\end{proposition}
We also have an analogue of Lemma~\ref{L-1} in the type~B setting.

\begin{lem}\label{L-3}
For $\pi\in\B_n^{>}$ (resp.\ $\pi\in\B_n^{<}$), there exists a unique tree 
$T({\pi}^\ast)\in\Orb^{*}(T_\pi)$, with ${\pi}^\ast\in\B_n^{>}$
(resp.\ ${\pi}^\ast\in\B_n^{<}$), such that
\[
\run({\pi}^\ast)=\leaf(T_\pi).
\]
\end{lem}

\begin{proof}
For $\pi\in\B_n^{>}$ (resp.\ $\pi\in\B_n^{<}$), if the node labeled $0$ is a leaf of $T_\pi$, then its parent must be $\pi_1$, and $\pi_1$ is the unique $\max$–node (resp.\ $\min$–node) of $T_\pi$ with two children.  
If the node labeled $0$ is an inner node, then it must be a $\min$–node (resp.\ a $\max$–node).  

Proceeding as in the proof of Lemma~\ref{L-1}, let  
\[
\pi_{i_1}\pi_{i_2}\cdots\pi_{i_k}
\]
be the sequence of labels of the nodes of $T_\pi$ with two children, listed in increasing order of positions: \(i_1<i_2<\cdots<i_k\).  
For $1\le j\le k$, by applying the MHR--action if necessary, we may transform $T_\pi$ into a tree  
\(T_{\pi'}\in\Orb^{*}(T_\pi)\), with $\pi'\in\B_n^{>}$ (resp.\ $\pi'\in\B_n^{<}$), such that
\[
\pi'_{i_1},\pi'_{i_2},\ldots,\pi'_{i_k}
\]
are alternatively max–nodes and min–nodes of $T_{\pi'}$ (resp.\ min–nodes and max–nodes).

Write
\[
\pi' = w_{1}\,\pi'_{i_1}\, w_{2}\,\pi'_{i_2}\,\cdots\, w_{k}\,\pi'_{i_k}\, w_{k+1},
\]
where each $w_j$ is (possibly empty) a consecutive block of letters of $\pi'$.

Now apply the MHR--action to $T_{\pi'}$ so that every node whose label lies in $w_j$ becomes a min–node (resp.\ a max–node) when $j$ is odd, and a max–node (resp.\ a min–node) when $j$ is even, for \(1\le j\le k+1\).  
Let $T_{{\pi}^\ast}$ be the resulting tree, with ${\pi}^\ast\in\B_n$.  
By the above construction, we obtain
\[
\run({\pi}^\ast) = \leaf(T_\pi).
\]

Finally, the tree $T_{{\pi}^\ast}$ is the unique element of $\Orb^{*}(T_\pi)$ satisfying this property:  
any other choice of MHR--actions on the blocks $w_j$ would necessarily produce at least one additional peak or valley in the corresponding permutation $\pi'$, thus increasing the number of alternating runs.
\end{proof}

Lemma \ref{L-3} also shows that ${\pi}^\ast$ is the only permutation with the shortest length of the longest alternating subsequence with $T_{{\pi}^\ast}\in\mathrm{Orb}^*(T_\pi)$. As in type A case, we use ${\pi}^\ast$ to denote this unique permutation with $T_{{\pi}^\ast}\in\mathrm{Orb}^*(T_\pi)$.


\begin{lemma}\label{L-4}
Let $\pi\in\B_n^{>}$ (resp.\ $\pi\in\B_n^{<}$), and let 
$T_{\pi^\ast}\in\Orb^{*}(T_\pi)$.  
If $S\subseteq [n-1]$ satisfies $\Lpv(\pi^\ast)\subseteq S$,  
then there exists a unique tree $T_{\sigma}\in\Orb^{*}(T_\pi)$ such that
\[
\Lpv(\sigma)=S.
\]
\end{lemma}

\begin{proof}
If the cardinality of $S$ is $k$, we can write it as $S=\{s_1<s_2<\cdots<s_k\}$. We then obtain $\sigma$ by the following process. Because $\sigma_0=0$, and $\pi\in\B_n^>$ (resp. $\pi\in\B_n^<$), that is $\pi_1>0$ (resp. $\pi_1<0$), which means that $s_1$ must be a left peak (resp. left valley). Since the valleys and peaks in a permutation are alternating, by using MHR action, we make $s_i$ a $\max$-node (resp. $\min$-node) if $i$ is odd, and make $s_i$ a $\min$-node (resp. $\max$-node) if $i$ is even. We also make every node with position between $s_{i-1}$ and $s_i$ a min (resp. max) node if $i$ is odd and make every node with position between $s_{i-1}$ and $s_i$ a max (resp. min) node if $i$ is even, here $s_0=0$ and $s_{k+1}=n+1$. Then the corresponding permutation of this HR--tree is the $\sigma$, and from this process we see that this is the only way to obtain a permutation $\sigma$ such that $T_\sigma\in\mathrm{Orb}^*(T_\pi)$ satisfying $\mathrm{Lpv(\sigma)=S}$ by using MHR-action. So, this completes the proof.
\end{proof}
\begin{theorem}\label{T-2}
For $\pi\in\B_n^>$ (resp. $\pi\in\B_n^<$), we have
\begin{align}\label{k-2}
\sum_{T_{\sigma}\in\mathrm{Orb}^*(T_\pi)}x^{\mathrm{run}_B(\sigma)}=x^{\leaf(T_\pi)}(1+x)^{n-\leaf(T_\pi)}.
\end{align}
\end{theorem}
\begin{proof}
By Lemma~\ref{L-3}, we have $\Lpv(\pi^\ast)\subseteq \Lpv(\sigma)$ for every 
$T_\sigma\in\Orb^{*}(T_\pi)$.  
Moreover, for any such $\sigma$,
\[
\run_B(\sigma)
   = |\Lpv(\sigma)|+1
   = \leaf(T_\pi)+1 \;+\; \bigl(|\Lpv(\sigma)|-\leaf(T_\pi)\bigr).
\]
The quantity $|\Lpv(\sigma)|-\leaf(T_\pi)$ precisely counts the additional
positions at which the MHR--action introduces a new left-peak or valley.  
Applying Lemma~\ref{L-4}, we obtain the desired identity~\eqref{k-2}.
\end{proof}

\subsection{Alternating runs in type B}

A permutation \(\pi\in \A_n^B\) is called an \emph{alternating permutation}
(down--up) if
\[
\pi_1>\pi_2<\pi_3>\pi_4<\cdots.
\]
An alternating permutation \(\pi\in\A_n^B\) is called a \emph{snake} if it starts
with a positive entry, that is, if \(\pi_1>0\).

Let \(\mathcal{DU}_n^{(B)}\) (resp.\ \(\mathcal{S}_n\)) denote the set of
alternating permutations (resp.\ snakes) of type~B of length \(n\).
Let \(DU_n^{(B)}=\lvert\mathcal{DU}_n^{(B)}\rvert\).
Clearly,
\[
DU_n^{(B)}=2^n E_n,
\]
where \(E_n\) is the Euler number.

Arnol'd~\cite{Ar92} proved that the cardinalities  of 
\(\mathcal{S}_n\) are  
the \textit{Springer numbers} $S_n$, which have the exponential generating function~\cite{Sp71},
\begin{align}\label{pan}
1+\sum_{n>0}S_n\frac{x^n}{n!}=\frac{1}{\cos(x)-\sin(x)}.
\end{align}


Define the sets
\begin{align}   
\mathcal{S}_{n,k}&=\{\pi\in\mathcal{S}_n\,|\,\leaf(T_\pi)=k\},\\
\mathcal{A}_{n,k}^B&=\{\pi\in\mathcal{A}_n^B\,|\,\des(\pi)=k\}.
\end{align}

\begin{lem}\label{s-a}
For $k\ge 1$, let $b_{n,k}$ and $\bar b_{n,k}$ denote the cardinalities of 
$\mathcal{S}_{n,k}$ and $\mathcal{A}_{n,k}^B$, respectively.  
Then
\[
b_{n,k}=\bar b_{n,k-1}.
\]
\end{lem}

\begin{proof}
We construct a bijection from $\mathcal{S}_{n,k}$ to $\mathcal{A}_{n,k-1}^B$.  
Let $\pi\in\mathcal{S}_{n,k}$.  
Then $T_\pi$ has $k-1$ inner nodes with two children, since it has $k$ leaves.  
We distinguish two cases according to the role of the node labeled~$0$ in $T_\pi$.

\smallskip
\emph{Case~1.}  
If the node labeled~$0$ is a leaf, then its parent, labeled~$\pi_1$,  
has two children and is necessarily a max–node, because $\pi_1>0$.  
Applying the HR--action, which turns every max–node of $T_\pi$ into a min–node,  
we obtain an HR–tree $T_\tau$.  
By Lemma~\ref{andre-B} and Fact~\ref{fact1}, the corresponding permutation $\tau$  
is a type~B André permutation with $k-1$ descents.

\smallskip
\emph{Case~2.}  
If the node labeled~$0$ has one child, then it must be a min–node,  
since its right subtree contains the node labeled $\pi_1$ and $0<\pi_1$.  
Applying the MHR--action, which converts all max–nodes of $T_\pi$ into min–nodes,  
we again obtain an HR–tree $T_\tau$.  
By Lemma~\ref{andre-B} and Fact~\ref{fact1},  
the associated permutation $\tau$ is a type~B André permutation with $k-1$ descents.

\smallskip
In both cases, the process is reversible, establishing a bijection between  
$\mathcal{S}_{n,k}$ and $\mathcal{A}_{n,k-1}^B$.  
\end{proof}

\begin{theorem}\label{CB-1}
For $n\ge 1$, define
\begin{subequations}
    \begin{equation}
    T_n(x)=\sum_{k=1}^{\lceil (n+1)/2 \rceil} b_{n,k}\, x^{k}=
\sum_{k=1}^{\lceil (n+1)/2 \rceil}\bar b_{n,k-1}\, x^{k}.
\end{equation}
Then
\begin{align}\label{chow-ma}
R_n^{B,>}(x)=R_n^{B,<}(x)
   = (1+x)^{n}\,
     T_n\!\left(\frac{x}{1+x}\right).
\end{align}
\end{subequations}

\end{theorem}

\begin{proof}
For $\pi\in\B_n^{>}$, the map $\pi\mapsto \overline{\pi}$ is a bijection from $\B_n^{>}$ to $\overline{\B_n^{<}}$ and satisfies  
$\run_B(\pi)=\run_B(\overline{\pi})$.  
This gives the first equality.  
By Lemma~\ref{L-4} with $S=[n-1]$, there is a unique snake $\sigma$ in $\Orb^{*}(T_\pi)$, and by Theorem~\ref{T-2} we obtain~\eqref{chow-ma}.
\end{proof}

\begin{remark}
Theorem~\ref{CB-1} provides two combinatorial interpretations of the polynomials $T_n(x)$, thereby answering a question of Chow and Ma~\cite[Section~5]{CM14}.
\end{remark}

The first few instances of  \(T_n(x)\) are 
\[
\begin{aligned}
T_1(x) &= x,\\
T_2(x) &= x + 2x^2,\\
T_3(x) &= x + 10x^2,\\
T_4(x) &= x + 36x^2 + 20x^3.
\end{aligned}
\]

For $T_\pi\in\bhr_n$, if the nodes with one child of $T_\pi$ are labeled by 
\[\pi_{i_1},\,\pi_{i_2},\,\cdots,\,\pi_{i_m}\quad  (0\leq i_1<i_2<\cdots<i_m), 
\]
then we say a node with one child of $T$ is \emph{the $k$-th node with one child}, if it is labeled by $\pi_{i_k}$, and we also call it an even (resp. odd) node with one child if $k$ is even (resp. odd).

\begin{definition}
Let \(\widetilde{\bhr}_n^{+}\) (resp.\ \(\widetilde{\bhr}_n^{-}\)) denote the subset of
trees in \(\bhr_n\) satisfying the following conditions:
\begin{enumerate}
\item every node with two children is a \(\max\)-node
      (resp.\ a \(\min\)-node);
\item every \(\max\)-node (resp.\ \(\min\)-node) with one child is an even node.
\end{enumerate}
\end{definition}

\begin{theorem}\label{Quotient-B}
 Let $m=\lfloor(n-1)/2\rfloor$ for $n\geq 1$, then
\begin{align}
R_n^{B,>}(x)&=(1+x)^{m}\sum_{T\in \widetilde{\bhr}_n^+}x^{\max(T)+1},\\
R_n^{B,<}(x)&=(1+x)^{m}\sum_{T\in \widetilde{\bhr}_n^-}x^{\min(T)+1},\\
R_n^{B}(x)&=(1+x)^{m}\left(\sum_{T\in \widetilde{\bhr}_n^+}x^{\max(T)+1}+\sum_{T\in \widetilde{\bhr}_n^-}x^{\min(T)+1}\right),
\end{align}
where $\max(T)$ (resp. $\min(T)$) is the number of $\max$-nodes ($\min$-nodes) of T.
\end{theorem}
\begin{proof}
Since $n-\lceil\frac{n+1}{2}\rceil=\lfloor\frac{n-1}{2}\rfloor$, by Theorem \ref{CB-1}, the polynomials  $R_n^{B,>}(x)$, 
$R_n^{B,<}(x)$ and  $R_n^{B}(x)$  are all divisible by $(1+x)^{\lfloor\frac{n-1}{2}\rfloor}$.
Moreover,  we have the following factorizations:
\begin{subequations}\label{eq:quotients B}
\begin{align}\label{Q-1}
R_n^{B,>}(x)&=(1+x)^{m}\sum_{k=1}^{\lceil\frac{n+1}{2}\rceil}b_{n,k}x^k(1+x)^{\lceil\frac{n+1}{2}\rceil-k},\\
 R_n^{B,<}(x)&=(1+x)^{m}\sum_{k=1}^{\lceil\frac{n+1}{2}\rceil}b_{n,k}x^k(1+x)^{\lceil\frac{n+1}{2}\rceil-k},\\
R_n^{B}(x)&=(1+x)^{m}\sum_{k=1}^{\lceil\frac{n+1}{2}\rceil}2\cdot b_{n,k}x^k(1+x)^{\lceil\frac{n+1}{2}\rceil-k},\label{Q-4}
\end{align}
\end{subequations}
where $m=\lfloor(n-1)/2\rfloor$.

For $T\in\bhr_n$ with $k$ leaves. If $0$ is a leaf of $T$, then $n-k-(k-1)$ is the number of nodes with one child of $T$ minus $1$. If $0$ is an inner node of $T$, then $0$ is the first node with one child, thus $n-k-(k-1)$ is also the number of nodes with one child of $T$ minus $1$. So we have $\lceil(n-k-(k-1))/2\rceil$ is the number of even nodes with one child of $T$. 
\end{proof}

\subsection{Link to type B Eulerian polynomials}
For $\pi=\pi_1\pi_2\ldots\pi_n\in\B_n$ with $\pi_0=0$, an index $i\in \{0, \ldots, n-1\}$ is a descent position of $\pi$ if $\pi_i>\pi_{i+1}$. The number of descents of $\pi$ is denoted by $\mathrm{des}_B\,\pi$. 
Then, the type B Eulerian polynomials are  defined by, see \cite{Br94, Pe15},
$$
B_n(x)=\sum_{\pi\in\B_n}x^{\mathrm{des}_B\,\pi}.
$$
Next, we define a group action on $\bhr_n$, which is based on the HR--action.
\begin{definition}[BHR--action]
Let $T_\pi\in\bhr_n$. For $0\leq i\leq n$, the operator $\widetilde{\psi}_i$ acting on the tree $T_\pi$ is defined as follows.
\begin{enumerate}
\item[(A)]
For $i\neq0$.
\begin{itemize}
\item 
If $\pi_i$ is a $\min$-node, then replace $\pi_i$ by the largest element of $T_{\pi}^r(\pi_i)$, permute the remaining elements of $T_{\pi}^r(\pi_i)$ such that they keep their same relative orders and all other nodes in $T_\pi$ are fixed.
\item
If $\pi_i$ is a $\max$-node, then replace $\pi_i$ by the smallest element of $T_{\pi}^r(\pi_i)$ such that they keep their same relative order, and all other nodes in $T_\pi$ are fixed.
\end{itemize}
\vskip 1.5 mm
\item[(B)]
For $i=0$ and $\pi_0$ is an inner node.
\begin{itemize}
\item
$\widetilde{\psi}_0(T_\pi)=\prod_{i=1}^n\widetilde{\psi}_i(T_{\overline{\pi}})$.
\end{itemize}
\end{enumerate}
\end{definition}
\begin{figure}[t]
\begin{tikzpicture}
\draw (0,0)--(1.5,1.5);
\draw (1,-1)--(0,0);
\draw (2,-1)--(3,0);
\draw (1.5,1.5)--(3,0);
\draw (4,-1)--(3,0);
\fill (0,0) circle (2pt);
\fill (1.5,1.5) circle (2pt);
\fill (1,-1) circle (2pt);
\fill (3,0) circle (2pt);
\fill (2,-1) circle (2pt);
\fill (4,-1) circle (2pt);
\node[left] at (0,0) {$0$};
\node[above] at (1.5,1.5) {$\bar{4}$};
\node[below] at (1,-1) {$3$};
\node[above] at (3,0) {$5$};
\node[below] at (2,-1) {$2$};
\node[below] at (4,-1) {$\bar{1}$};
\draw[][->,very thick](4,0)--(6,0) node[midway, above]{$\widetilde{\psi}_0$};
\draw (7,0)--(8.5,1.5);
\draw (8,-1)--(7,0);
\draw (9,-1)--(10,0);
\draw (8.5,1.5)--(10,0);
\draw (11,-1)--(10,0);
\fill (7,0) circle (2pt);
\fill (8.5,1.5) circle (2pt);
\fill (8,-1) circle (2pt);
\fill (10,0) circle (2pt);
\fill (9,-1) circle (2pt);
\fill (11,-1) circle (2pt);
\node[left] at (7,0) {$0$};
\node[above] at (8.5,1.5) {$\bar{5}$};
\node[below] at (8,-1) {$\bar{3}$};
\node[above] at (10,0) {$4$};
\node[below] at (9,-1) {$1$};
\node[below] at (11,-1) {$\bar{2}$};
\node[below] at (1.5,-2) {$T_{03\bar{4}25\bar{1}}$};
\node[below] at (8.5,-2) {$T_{0\bar{3}\bar{5}14\bar{2}}$};
\end{tikzpicture}
    \caption{A BHR--action:  $\widetilde{\psi}_0(T_{03\bar{4}25\bar{1}})=\prod_{i=1}^5\widetilde{\psi}_i(T_{0\bar{3}{4}\bar{2}\bar{5}{1}})=T_{0\bar{3}\bar{5}14\bar{2}}$.}\label{F-4}
\end{figure}
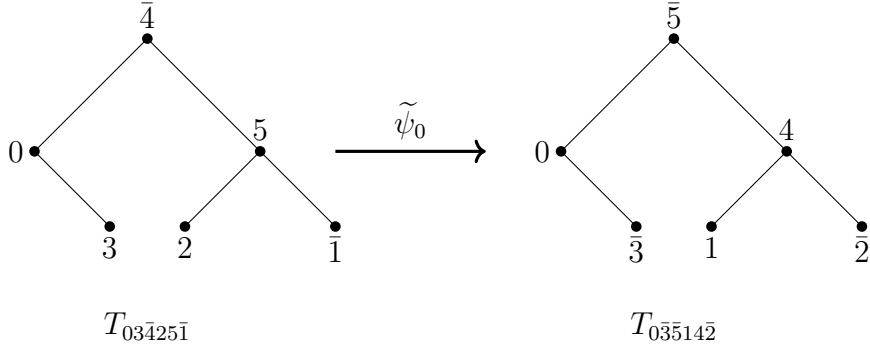

Note that for $T_\pi\in\bhr_n$, if $\pi_0$ is an inner node, then $\widetilde{\psi}_0(T_\pi)=T_\sigma$ and $T_\pi$ have the same shape. For $1\leq i\leq n$, the node of $T_\sigma$ labeled by $\sigma_i$ and the node of $T_\pi$ labeled by $\pi_i$ are either both $\max$-nodes or both $\min$-nodes. And if $0$ is a $\min$-node (resp. $\max$-node) in $T_\sigma$ then $0$ is a $\max$-node (resp. $\min$-node) in $T_\pi$.

For example, Figure~\ref{F-4} gives an illustration of BHR--action in case (B) with 
\[
\pi=03\bar{4}25\bar{1},  \quad 
\overline{\pi}=0\bar{3}{4}\bar{2}\bar{5}{1}.
\]

\par
It is clear that $\widetilde{\psi}_i$ is an involution acting on $\bhr_n$ and that $\widetilde{\psi}_i$ and $\widetilde{\psi}_j$ commute for all $i,\,j\in[0,\,n]$. Hence, for any subset $S\subseteq[0,\,n]$ we may define the map
\[\widetilde{\psi}_S:\,\bhr_n\longrightarrow \bhr_n,\qquad
\widetilde{\psi}_S(T_\pi)=\prod_{i\in S}\widetilde{\psi}_i(T_\pi).
\]
The group $\mathbb{Z}_2^n$ acts on $\bhr_n$ via the maps $\widetilde{\psi}_S$ for  $S\subseteq[0,\,n]$. For $\pi\in\B_n$, let $\mathrm{Orb}^B(T_\pi)=\{g(T_\pi):g\in\mathbb{Z}_2^n\}$ be the orbit of $T_\pi$ under the \emph{BHR--action}. We have the following proposition regarding BHR--action.
\begin{proposition}
For $T_\pi\in\bhr_n$, there is one and only one type B Andr\'e permutations  in $\mathrm{Orb}^B(T_\pi)$.
\end{proposition}
\begin{proof}
By BHR--action, there is one and only one $T\in\mathrm{Orb}^B(T_\pi)$ such that all its inner nodes are $\min$-nodes. Hence, by Lemma \ref{andre-B}, the proposition holds.
\end{proof}

 The following $\gamma$-positive formula is due to Petersen~\cite[Theorem~13.5]{Pe15},
 \begin{align}\label{petersen-gamma}
 B_n(x)=\sum_{j=0}^{\lfloor n/2\rfloor}\gamma_{n,j}^{B}x^j(1+x)^{n-2j},
\end{align}
 where $\gamma_{n,j}^{B}=4^j\cdot|\{\sigma\in\S_n: |\Lpk(\sigma)|=j\}|$.
 
For example, the first few identities are 
 \begin{align*}
B_{1}(t)&=1+t;\\
B_{2}(t)&=(1+t)^{2}+4t;\\
 B_{3}(t)&=(1+t)^{3}+20 t(1+t);\\
B_{4}(t)&=(1+t)^{4}+72 t(1+t)^{2}+80t^{2}.
\end{align*}

For our purpose, we need the following interpretation of the $\gamma$-coefficients.
\begin{theorem}\label{gamma-B}
For all $n\geq1$, we have
\begin{align}
B_n(x)=\sum_{j=0}^{\lfloor n/2\rfloor}2^j\cdot\hat{b}_{n,j}x^j(1+x)^{n-2j},
\end{align}
where $\hat{b}_{n,j}=\#\{\pi\in\mathcal{A}_n^B|\mathrm{val}(\pi)=j\}$.
\end{theorem}
\begin{proof}
Let $\pi\in\A_n^B$ and $T_\pi$ is the corresponding BHR--tree. We have the following observations:
\begin{itemize}
    \item
    If $i$ is a valley of $\pi$, i.e., $\pi_{i-1}>\pi_i<\pi_{i+1}$ for
    $i\in [n-1]$, then the node labeled by $\pi_i$ of $T_\pi$ is an inner node with two children.
    \item 
    If $\pi_i$ is a label of an inner node with two children of $T_\pi$, then $i-1$ is a descent position of $\pi$, $i$ is a descent position of $w(\widetilde{\psi}_i(\pi))$, and $\mathrm{des}_B\,w(\widetilde{\psi}_i(\pi))=\mathrm{des}_B\,\pi$.
    \item 
    If $\pi_i$ is a label of an inner node with one child of $T_\pi$, then $i$ is an ascent position of $\pi$, $i$ is a descent position of $w(\widetilde{\psi}_i(\pi))$,and $\mathrm{des}_B\,w(\widetilde{\psi}_i(\pi))=\mathrm{des}_B\,\pi+1$.
\end{itemize}
Let
\begin{align}
c_i(\pi)=
\begin{cases}
 2x,& \text{if $\pi_i$ is a label of an inner node with two children in $T_\pi$};\nonumber\\
 1+x,&\text{if $\pi_i$ is a label of an inner node with one child in $T_\pi$};\\
 1,&\text{if $\pi_i$ is a label of a leaf in $T_\pi$.}\\
 \end{cases}
 \end{align} 
 If $T_\pi$ has  
  $j$  inner nodes with 2 children, then  $\leaf(T_\pi)=j+1$, and $T_\pi$ has $n-2j$ inner nodes with one child.  Applying the above observations we have
\begin{align*}
\sum_{T\in\mathrm{Orb}^B(T_\pi)}x^{\mathrm{des}_B\,w(T)}&=c_0(\pi)c_1(\pi)\cdots c_n(\pi)\\
&=(2x)^j(1+x)^{n-2j}.
\end{align*}
Summing over all $\pi\in\A_n^B$ we  obtain the result.
\end{proof}

\begin{remark}
    Comparing with Petersen's formula~\eqref{petersen-gamma} we have $2^j\bar b_{n,j}=\gamma_{n,j}^B$, which yields the recurrence relation~(see \cite{Ch08} and A008971 in \emph{OEIS}),
\begin{equation}
    \bar{b}_{n,k}=(1+2k)\bar{b}_{n-1,k}+(2n-4k+2)\bar{b}_{n-1,k-1}.
\end{equation}
\end{remark}

We have 
the following type B  David-Barton formula.
\begin{theorem}  
Let $w=\sqrt{\frac{1-x}{1+x}}$.
For $n\geq 1$ we have
\begin{align}\label{Type B David-Barton} 
R^{B,>}_n(x)=\frac{x}{2}\left(\frac{1+x}{2}\right)^{n-1} (1+w)^n B_n\left(\frac{1-w}{1+w}\right).
\end{align}
\end{theorem}
\begin{proof}
    Theorem~\ref{gamma-B} implies that 
\begin{align*}
(1+w)^nB_n\left(\frac{1-w}{1+w}\right)
=\sum_{j=0}^{\lfloor n/2\rfloor}2^{n}\cdot\bar{b}_{n,j}\left(\frac{x}{1+x}\right)^j,
\end{align*}
which is equivalent to
\begin{align}
\frac{x}{1+x}\left(\frac{1+x}{2}\right)^n(1+w)^nB_n\left(\frac{1-w}{1+w}\right)=\sum_{j=0}^{\lfloor n/2\rfloor}\bar{b}_{n,j}x^{j+1}(1+x)^{n-j-1}.
\end{align}
Applying Theorem~\ref{CB-1} yields the desired formula.
\end{proof}
\begin{remark}
Formula~\eqref{Type B David-Barton} (with a minor typo) already appeared in 
Zhao~\cite[Theorem~4.3.3]{Zh11}, although no proof was provided there.  
Ma--Ma--Yeh~\cite[p.~8]{MMY20} established an equivalent form of this formula 
by combining Petersen’s $\gamma$--formula~\eqref{petersen-gamma} with a result 
of Zhao~\cite[Theorem~4.3.1]{Zh11}.
\end{remark}

Chow and Ma~\cite[p. 56]{CM14} asked for a combinatorial proof of the following identity~\cite[Theorem~13, Eq. (19)]{CM14}:
\begin{equation}\label{CMopen1}
    2^{n-1} R_{n+1}(x)=2R_{n}^{B,>}(x)+\frac{1+x}{x}\sum_{k=1}^{n-1}\binom{n}{k}R_{k}^{B,>}(x)\,R_{n-k}^{B,>}(x),
\end{equation}

Combining the type~A and type~B David--Barton formulas 
\eqref{Run-Euler} and~\eqref{Type B David-Barton}, and performing the substitution 
\(z \mapsto \tfrac{1-w}{1+w}\) with \(w=\sqrt{\tfrac{1-x}{1+x}}\), 
Eq.~\eqref{CMopen1} becomes equivalent to the following identity 
relating the type~A and type~B Eulerian polynomials.  
We now provide a combinatorial proof of this identity.

\begin{theorem}
     Let $n\geq 1$. Then 
\begin{align}\label{eq:Chow-Ma}
    2^n\,A_{n+1}(z) /z=\sum_{k=0}^n\binom{n}{k}B_k(z)\,B_{n-k}(z).
\end{align}
\end{theorem}
\begin{proof}
For a signed permutation \(\pi=\pi_1\cdots\pi_n\in \B_n\) set \(\pi_0:=0\) and
\[
\des_B(\pi):=\#\{\,i\in\{0,1,\dots,n-1\}:\ \pi_i>\pi_{i+1}\,\}
\]
with respect to the usual total order on integers.
For an unsigned permutation we write \(\des(\cdot)\) for the usual type~A
descent number.

Let
\[
\hat{\B}_{n+1}:=\{\pi\in \B_{n+1}:\ \pi(i)=1\ \text{for some }i\}
\qquad\text{and}\qquad
\hat B_{n+1}(z):=\sum_{\pi\in\hat{\B}_{n+1}} z^{\des(\pi)} .
\]
We interpret \(\des(\pi)\) for \(\pi\in\hat{\B}_{n+1}\) as the type~A descent number
of the word \(\pi_1\cdots\pi_{n+1}\) read in the usual order on \(\mathbb Z\).

Fix \(\pi\in\hat{\B}_{n+1}\) and write it uniquely as
\(\pi=\pi' \, 1 \, \pi''\), where \(\pi'\) has length \(k\) and \(\pi''\) has length \(n-k\)
(\(0\le k\le n\)).
Then
\[
\des(\pi)=\des(\pi'1)+\des(1\pi''),
\]
since all potential descents of \(\pi\) occur either inside \(\pi'\), at the cut
\(\pi'_k\vert 1\), at the cut \(1\vert \pi''_1\), or inside \(\pi''\), with no overlap.

\smallskip
\emph{Right part.}
Define \(\tau:=1\pi''=\tau_1\cdots\tau_{n-k+1}\), and from \(\tau\) obtain
\(\tau'\) by shifting labels across \(1\):
\[
\tau'_i=\begin{cases}
\tau_i-1,& \tau_i>0,\\
\tau_i+1,& \tau_i<0.
\end{cases}
\]
Let \(\tau'':=\tau'_2\cdots\tau'_{n-k+1}\in\B_{n-k}\).
The map \(\tau\mapsto\tau'\) is order-preserving on \(\mathbb Z\setminus\{1\}\)
and sends the initial letter \(1\) to \(0\). Hence
\[
\des(1\pi'')=\des(\tau)=\des(\tau')=\des_B(\tau'').
\]

\smallskip
\emph{Left part.}
Similarly, set \(\sigma:=\pi'1\) and form \(\sigma'\) by the same shift across \(1\).
Now define
\[
\sigma'':=\overline{\sigma'_k}\ \overline{\sigma'_{k-1}}\ \cdots\ \overline{\sigma'_1}\in \B_k,
\]
where the bar means changing the sign. 
The composition “reverse + sign-flip’’ is an order anti-automorphism on
\(\mathbb Z\setminus\{0\}\), and inserting the leading \(0\) for type \(B\) transforms
descents of \(\sigma'\) into type \(B\) descents of \(\sigma''\). Consequently,
\[
\des(\pi'1)=\des(\sigma)=\des(\sigma')=\des_B(\sigma'').
\]

Putting the two parts together gives
\[
\des(\pi)=\des_B(\sigma'')+\des_B(\tau'').
\]
For a fixed split size \(k\), the choice of the \(k\) absolute values that appear
to the left of \(1\) contributes a factor \(\binom{n}{k}\), and \(\sigma''\) and \(\tau''\)
range freely over \(\B_k\) and \(\B_{n-k}\), respectively. Therefore
\[
\hat B_{n+1}(z)=\sum_{k=0}^n \binom{n}{k}\, B_k(z)\, B_{n-k}(z).
\]

On the other hand, forgetting the signs of the entries \(\neq 1\) gives a
\(2^n\)-to-\(1\) map \(\hat{\B}_{n+1}\to \mathfrak S_{n+1}\) that preserves type~A
descents (changing signs does not affect the relative order between any two
nonzero absolute values or with \(1\)). Hence
\[
\hat B_{n+1}(z)=2^n \sum_{\sigma\in \mathfrak S_{n+1}} z^{\des(\sigma)}=2^n A_{n+1}(z)/z.
\]
Combining the two expressions for \(\hat B_{n+1}(z)\) yields the desired identity \eqref{eq:Chow-Ma}.
\end{proof}

\begin{remark}
    Sun~\cite{Su21} proved a $\gamma$-positive expansion for the bivariate type~A Eulerian polynomials refined by  the parity of descent positions. 
A type B analogue of Sun's formula has been given by two of the present  authors in \cite{PZ23}.  
Using the BHR--action, one can similarly 
refine Theorem~\ref{gamma-B} with respect to the  parity of descent positions.
\end{remark}


\section{Alternating runs in type D}\label{sec:typeD}

Let \(\D_n\subseteq\B_n\) denote the subset of signed permutations having an
even number of negative entries.  
For \(\pi\in\D_n\), we use the same notion of alternating runs as in \(\B_n\); in
particular,
\[
\run_D(\pi)=\run_B(\pi)\qquad\text{for all }\pi\in\D_n.
\]
Define the following polynomials:
\begin{align*}
R_n^D(x)&=\sum_{\pi\in\D_n}x^{\run_D(\pi)},
\;\;R_n^{D,>}(x)=\sum_{\pi\in\D_n^>}x^{\run_D(\pi)},
\;\;R_n^{D,<}(x)=\sum_{\pi\in\D_n^<}x^{\run_D(\pi)},
\end{align*}
where 
\[
\D_n^>=\{\pi\in\D_n| \pi_1>0\}, \qquad
\D_n^<=\{\pi\in\D_n^<| \pi_1<0\}.
\]
The first few polynomials in these three families are listed in
Table~\ref{table-3}.

\begin{table}[ht]
\begin{tabular}{c|c|c|c}
&$R_n^D(x)$&$R_n^{D,>}(x)$&$R_n^{D,<}(x)$\\
\hline
$n=2$&$2x+2x^2$&$x+x^2$&$x+x^2$\\
$n=3$&$x+12x^2+11x^3$&$x+6x^2+5x^3$&$6x^2+6x^3$\\
$n=4$&$2x+38x^2+94x^3+58x^4$&$x+19x^2+47x^3+29x^4$&$x+19x^2+47x^3+29x^4$\\
\hline
\end{tabular}
\vspace{10pt}
\caption{First terms of $R_n^D(x)$, $R_n^{D,>}(x)$ and $R_n^{D,<}(x)$.}
\label{table-3}
\end{table}
Since \(\D_n^{>}\subseteq\B_n^{>}\) and \(\D_n^{<}\subseteq\B_n^{<}\), any statement
proved for all \(\pi\in\B_n^{>}\) (resp.\ \(\pi\in\B_n^{<}\)) remains valid for all
\(\pi\in\D_n^{>}\) (resp.\ \(\pi\in\D_n^{<}\)).
Define
\[
\mathcal{A}_n^D=\mathcal{A}_n^B\cap \D_n .
\]

\begin{defi}
A permutation in $\mathcal{A}_n^D$ is called a type D Andr\'e permutations of length $n$. 
\end{defi}
 Since the MHR--action does not change the number of negative terms in a signed permutation, by Theorem \ref{CB-1}  we have the following results.
\begin{theorem}\label{CD-1}
For $n\geq1$, we have
\begin{align}\label{k-6}
R_n^{D,>}(x)&=\sum_{k=1}^{\lceil\frac{n+1}{2}\rceil}\widehat{d}_{n,k}x^k(1+x)^{n-k}\\
&=\sum_{k=0}^{\lfloor\frac{n}{2}\rfloor}\tilde{d}_{n,k}x^{k-1}(1+x)^{n+1-k},
\end{align}
where $\widehat{d}_{n,k}=\#\{\pi\in\mathcal{S}_n | \leaf(T_\pi)=k\; \text{and}\;\, |\Neg(\pi)|\equiv0\; (\mathrm{mod}\;2)\}$, and $\tilde{d}_{n,k}$ is the number of type D Andr\'e permutations  with length $n$ and $k$ valleys.
\end{theorem}
\begin{theorem}\label{CD-2}
For $n\geq1$, we have
\begin{align}\label{k-7}
R_n^{D,<}(x)=\sum_{k=1}^{\lceil\frac{n+1}{2}\rceil}\bar{d}_{n,k}x^k(1+x)^{n-k},
\end{align}
where $\bar{d}_{n,k}=\#\{\pi\in\overline{\mathcal{S}}_n | \leaf(T_\pi)=k\; \text{and}\;\, |\Neg(\pi)|\equiv0\; (\mathrm{mod}\;2)\}$.
\end{theorem}
Note that  $\widehat{d}_{n,k}=\bar{d}_{n,k}$ if and only if 
$n$ is even. 

\begin{coro}\label{CD-3}
For $n\geq1$, we have
\begin{align}\label{k-8}
R_n^{D}(x)=\sum_{k=1}^{\lceil\frac{n+1}{2}\rceil}(\widehat{d}_{n,k}+\bar{d}_{n,k})x^k(1+x)^{n-k},
\end{align}
where $\widehat{d}_{n,k}+\bar{d}_{n,k}=\#\{\pi\in\overline{\mathcal{S}}_n\cup\mathcal{S}_n | \leaf(T_\pi)=k\; \text{and}\;\, |\Neg(\pi)|\equiv0\; (\mathrm{mod}\;2) \}$.
\end{coro}
\vskip 0.3cm
The first few such decompositions of $R_n^D(x)$ ($n\geq 2$) are :
\begin{align*}
R_2^D(x)&=2x(1+x),\\[4pt]
R_3^D(x)&=x(1+x)^2+10x^2(1+x),\\[4pt]
R_4^D(x)&=2x(1+x)^3+32x^2(1+x)^2+24x^3(1+x),\\[4pt]
R_5^D(x)&=x(1+x)^4+116x^2(1+x)^3+244x^3(1+x)^2,\\[4pt]
R_6^D(x)&=2x(1+x)^5+332x^2(1+x)^4+1804x^3(1+x)^3+448x^4(1+x)^2.
\end{align*}

\begin{theorem}
For $n\geq 2$, let $m=\lfloor(n-1)/2\rfloor$, and let  $\widetilde{\dhr}_n^+$ (resp. $\widetilde{\dhr}_n^-$) denote  the set  of HR--trees in $\widetilde{\bhr}_n$ (resp. $\widetilde{\bhr}_n^-$) with even numbers of negative nodes. Then
\begin{subequations}
\begin{align}
R_n^{D,>}(x)&=(1+x)^{m}\sum_{T\in \widetilde{\dhr}_n^+}x^{\max(T)+1},\\
R_n^{D,<}(x)&=(1+x)^{m}\sum_{T\in \widetilde{\dhr}_n^-}x^{\min(T)+1},\\
R_n^{D}(x)&=(1+x)^{m}\left(\sum_{T\in \widetilde{\dhr}_n^+}x^{\max(T)+1}+\sum_{T\in \widetilde{\dhr}_n^-}x^{\min(T)+1}\right),
\end{align}
\end{subequations}
where $\max(T)$ (resp. $\min(T)$) is the number of $\max$-nodes ($\min$-nodes) of T.
\end{theorem}
\begin{proof} Let $m=\lfloor (n-1)/2\rfloor$.
Since 
\[
n-\Big\lceil \frac{n+1}{2}\Big\rceil = \Big\lfloor \frac{n-1}{2}\Big\rfloor,
\]
Theorem~\ref{CD-1}, Theorem~\ref{CD-2}, and Corollary~\ref{CD-3} imply the following product formulae:
\begin{subequations}\label{eq:quotients D}
\begin{align}\label{Q-2}
R_n^{D,>}(x)&=(1+x)^{m}\sum_{k=1}^{\lceil\frac{n+1}{2}\rceil}\widehat{d}_{n,k}x^k(1+x)^{\lceil\frac{n+1}{2}\rceil-k},\\
\label{Q-5}
R_n^{D,<}(x)&=(1+x)^{m}\sum_{k=1}^{\lceil\frac{n+1}{2}\rceil}\bar{d}_{n,k}x^k(1+x)^{\lceil\frac{n+1}{2}\rceil-k},\\
R_n^{D}(x)&=(1+x)^{m}\sum_{k=1}^{\lceil\frac{n+1}{2}\rceil}(\widehat{d}_{n,k}+\bar{d}_{n,k})x^k(1+x)^{\lceil\frac{n+1}{2}\rceil-k}.
\end{align}
\end{subequations}
By Theorem \ref{Quotient-B}  the desired results follow.
\end{proof}
Let $Q_n^B(x)=R_n^B(x)(1+x)^{-m}$ and $Q_n^D(x)=R_n^D(x)(1+x)^{-m}$, where $m=\lfloor(n-1)/2\rfloor$. Then
\begin{align} 
Q_n^B(-1)&=2b_{n,k}(-1)^{\lceil\frac{n+1}{2}\rceil}\neq0;\\ 
Q_n^D(-1)&=(\widehat{d}_{n,k}+\bar{d}_{n,k})(-1)^{\lceil\frac{n+1}{2}\rceil}\neq0.
\end{align} 
Thus $-1$ is a zero of multilicity $m$ of $R_n^B(x)$ and $R_n^D(x)$. Actually, it is also easy to see that $-1$ is a zero of multilicity $m$ of $R_n^{B, >}(x)$, $R_n^{B, <}(x)$, $R_n^{D,>}(x)$ and $R_n^{D, <}(x)$.

\begin{remark} 
In \cite[Theorem~6.3]{DPS09} the following identity is proved
\[
B_n(t)=D_n(t)+n 2^{n-1} A_{n-2}(t) \quad (n\geq 2).
\]
There is a type D $\gamma$-positive formula in \cite[Theorem~13.9]{Pe15}.
It would be interesting to know whether there is a suitable  
David-Barton type formula of type D.
\end{remark}
\section*{Acknowledgement}
The first author's work was supported by the National Natural Science Foundation of China grant 12201468.


\end{document}